\documentclass[11pt, reqno]{article}

\usepackage{fullpage}
\usepackage{enumerate}
\usepackage{setspace}

\usepackage{amsmath,amsfonts,amssymb,graphics,amsthm, comment}
\usepackage{hyperref}
\hypersetup{
    colorlinks=true,
    linkcolor=blue,
    citecolor=red,
    urlcolor=blue,
    pdfborder={0 0 0}
}

\usepackage[T1]{fontenc}
\usepackage{mathabx,graphicx}

\usepackage{graphicx}
\usepackage{xcolor}
\usepackage{bbm}
\usepackage{wrapfig} 

\usepackage[font=sf, labelfont={sf,bf}, margin=1cm]{caption}

\usepackage{cleveref}
  \crefname{theorem}{Theorem}{Theorems}
  \crefname{thm}{Theorem}{Theorems}
  \crefname{lemma}{Lemma}{Lemmas}
  \crefname{lem}{Lemma}{Lemmas}
  \crefname{remark}{Remark}{Remarks}
  \crefname{prop}{Proposition}{Propositions}
\crefname{notation}{Notation}{Notations}
\crefname{claim}{Claim}{Claims}
  \crefname{defn}{Definition}{Definitions}
  \crefname{corollary}{Corollary}{Corollaries}
  \crefname{section}{Section}{Sections}
  \crefname{figure}{Figure}{Figures}
    \crefname{assumption}{Assumption}{Assumptions}

\newtheorem{thm}{Theorem}[section]

\newtheorem{lemma}[thm]{Lemma}
\newtheorem{corollary}[thm]{Corollary}
\newtheorem{prop}[thm]{Proposition}
\newtheorem{defn}[thm]{Definition}

\newtheorem{question}[thm]{Question}

\newtheorem{assumption}[thm]{Assumptions}
\numberwithin{equation}{section}

\theoremstyle{definition}
\newtheorem{remark}[thm]{Remark}

\def \1  {(1)}
\def \2 {(2)}

\def \ve {\varepsilon}

\def\P{\mathbb{P}}
\def\E{\mathbb{E}}
\def\C{\mathbb{C}}
\def\R{\mathbb{R}}

\def\D{\mathbb{D}}

\def\H{\mathbb{H}}

\def  \p- {p\textunderscore}

\def\eps{\varepsilon}

\def\e{\text{e}}
\def\ph{\varphi}

\DeclareMathOperator{\gff}{GFF}

\def \sou {\mathsf D_u}
\def \sor {\mathsf D_r}
\def \sos {\mathsf D_s}

\def  \ve {\varepsilon}
\DeclareMathOperator{\supp}{Support}

\RequirePackage[normalem]{ulem} 
\RequirePackage{color}\definecolor{RED}{rgb}{1,0,0}\definecolor{BLUE}{rgb}{0,0,1} 

\title{$(1+\eps)$ moments suffice to characterise the GFF }
\author{Nathana\"el Berestycki\thanks{Supported in part by EPSRC grant EP/L018896/1, the University of Vienna, and FWF grant ``Scaling limits in random conformal geometry''.} \and Ellen Powell \and Gourab Ray\thanks{Supported in part by NSERC 50311-57400 and University of Victoria start-up 10000-27458}}

\begin{document}

\maketitle

\begin{abstract}
We show that there is ``no stable free field of index $\alpha\in (1,2)$'', in the following sense.
It was proved in \cite{BPR18} that subject to a \emph{fourth moment assumption}, any random generalised function on a domain $D$ of the plane, satisfying conformal invariance and a natural domain Markov property, must be a constant multiple of the Gaussian free field. In this article we show that the existence of $(1+\ve)$ moments is sufficient for the same conclusion. A key idea is a new way of exploring the field, where (instead of looking at the more standard circle averages) we start from the boundary and discover averages of the field with respect to a certain ``hitting density'' of It\^o excursions.
\end{abstract}

\section{Introduction}

The \textbf{Gaussian free field} (GFF) is a universal object believed (and in many cases proved) to govern the fluctuation statistics of many natural
 random surface models \cite{GOS, NS, MillerGL,Kenyon_GFF, dubedat_torsion, BLRdimers,BLRtorus,DubedatGheissari,Li} (see, e.g., \cite{LQGnotes,LNWP} for an introduction and survey of some recent developments). Although the GFF can be defined in any dimension, this article is concerned with the planar continuum version, which satisfies two special properties; namely, \textbf{conformal invariance} and a \textbf{domain Markov property}. The former roughly entails that applying a conformal map to a GFF in any domain produces a GFF in the image domain. The latter says, informally, that for any $D' \subset D \subset \C$, the conditional law of the GFF on $D$ restricted to $D'$, given its behaviour outside of $D'$, is that of the harmonic extension of the GFF from $\partial D'$ to $D'$ plus an independent GFF in $D'$. However, one major technical issue with defining the GFF is that it cannot be made sense of as a random function. It is instead defined as a random generalised function, which in this article we view as a stochastic process indexed by smooth, compactly supported test functions. As a result, some preparation is required in order to rigorously formulate the above properties.

We will now formally state our assumptions, which are essentially the same as in \cite{BPR18} except for the moment condition and the Dirichlet\footnote{We use the terminology ``Dirichlet'' and ``zero'' boundary conditions for the same notion throughout} boundary condition (we will comment after the theorem on the necessity of this adaptation).

Assume that for every simply connected domain $D\subset \C$, a stochastic process $h^D = (h^D_\phi)_{\phi \in C_c^\infty(D) }$ indexed by test functions is given. Assume further that each $h^D$ is linear in $\phi$: that is, for any $\lambda, \mu \in \R$ and $\phi, \phi'\in C_c^\infty(D)$, $$h^D_{\lambda \phi + \mu \phi'} = \lambda h^D_\phi + \mu h^D_{\phi'} \text{ almost surely. }$$ We then write, with an abuse of notation,
$$(
h^D, \phi) := h^D_\phi \text{ for } \phi \in C_c^\infty(D).
$$
We denote by $\Gamma^D $ the law of the stochastic process $h^D$. Thus $\Gamma^D$ is a probability distribution on $\R^{C_c^\infty(D)}$ equipped with the product topology. By Kolmogorov's extension theorem $\Gamma^D$ is characterised by its consistent finite-dimensional distributions: i.e., by the joint law of $(h^D, \phi_1), \ldots, (h^D, \phi_k)$ for any $k \ge 1$ and any $\phi_1, \ldots, \phi_k \in C_c^\infty(D)$.

We finally recall that the $H^{-1}(D)$ norm of a function $f\in C_c^\infty(D)$ is given by \begin{equation}
	\label{eqn:hminus1}
(f,f)_{-1}:=((-\Delta)^{-1/2}f,(-\Delta)^{-1/2}f)=(f,(-\Delta^{-1})f)= \iint_{D\times D} G_D(x,y) f(x)f(y) \, dx dy	\end{equation}
where $G_D$ is the Green function with Dirichlet boundary conditions in $D$. 

In the following, we write $\D=\{w\in \C: |w|<1\}$ and for $z\in \C$, $\eps>0$, we set $B_z(\eps):=\{w\in \C: |w-z|<\eps \}$. When $z$ lies in an open set $U\subset \C$, we write $d(z,\partial U):=\inf_{y\in \partial U} |y-z|$.

Let $D \subset \C$ be a proper simply connected open domain, and let $h^D$ be a sample from $\Gamma^D$.
\begin{assumption} We make the following assumptions.
\label{ass:ci_dmp}
\begin{enumerate}[(i)]

\item \textbf{(Moments)} For every $\phi\in C_c^\infty(D)$ and some $\xi>1$:
$$\E[(h^D,\phi)]=0 \;\; \text{and} \;\; \E[|(h^D,\phi)|^\xi]<\infty. $$

\item \textbf{(Continuity and Dirichlet boundary conditions)}  If $\phi_n\to \phi$ in $C_c^\infty(D)$, then $(h^D,\phi_n)\to (h^D,\phi)$ in probability as $n\to \infty$. Moreover, suppose that $(\phi_n)_{n \ge 1}$ is a sequence of non-negative test functions in $C_c^\infty(D)$, such that $d_n:=\sup\{d(z,\partial D)\, : \, z\in\text{Support}(\phi_n)\}\to 0$ as $n\to \infty$, and $\phi_n \to 0$ in $H^{-1}(D)$.
Then we have that $(h^D,\phi_n) \to 0$ in probability and in $L^1$ as $n\to \infty$.

\item \textbf{(Conformal invariance.)} Let $f: D \to D'$ be a bijective conformal map. Then
$
\Gamma^{D}  = \Gamma^{D'}\circ f,
$
where $\Gamma^{D'} \circ f$ is the law of the stochastic process $(h^{D'}, |(f^{-1})'|^2 (\phi \circ f^{-1}))_{\phi \in C_c^\infty(D)}$.

\item \textbf{(Domain Markov property)}. Suppose $D' \subset D$ is a simply connected Jordan domain. Then  we can decompose
$
h^D=  h^{D'}_D+\ph_D^{D'},
$
where:
\begin{itemize}
	\item $ h^{D'}_D$ is independent of $\ph_D^{D'}$;
	\item $(\ph_D^{D'},\phi)_{\phi\in C_c^\infty(D)}$ is a stochastic process indexed by $C_c^\infty(D)$ that is a.s. linear in $\phi$ and such that when we restrict to $C_c^\infty(D')$, $$(\ph_D^{D'},\phi)_{\phi\in C_c^\infty(D')}$$ a.s. corresponds to integrating against a harmonic function in $D'$. 
	\item $((h^{D'}_D,\phi))_{\phi\in C_c^\infty(D)}$ is a stochastic process indexed by $C_c^\infty(D)$, such that $(h^{D'}_D,\phi)_{\phi\in C_c^\infty(D')}$ has law $\Gamma^{D'}$ and $(h^{D'}_D,\phi)=0$ a.s. for any $\phi$ with $\supp(\phi)\subset D\setminus D'$.
\end{itemize}
\end{enumerate}
\end{assumption}

Observe that in light of \textit{(iii)}, the Dirichlet boundary condition \textit{(ii)} holds in one simply connected domain $D$ if and only if it holds in all simply connected domains. Indeed, suppose that it holds in $D$ and let $f:D\to D'$ be a conformal map. Then if $(\phi_n)_{n}\to 0\in H^{-1}(D')$, we have by conformal invariance of the Green function that $\tilde{\phi}_n:=|f|^2 (\phi_n\circ f)$ converges to $0$ in $H^{-1}(D)$, and since $(h^{D'},\phi_n)$ is equal in law to $(h^D,\tilde{\phi}_n)$, that $(h^{D'},\phi_n)\to 0$ in probability and in $L^1$ as $n\to \infty$.

We now comment on the main changes with respect to the assumptions in \cite{BPR18}. As already mentioned, the main change is the fact that we have replaced a moment of order four in (i) with a moment of order $\xi$ where $\xi>1$.
Beyond this, we have slightly adapted the Dirichlet boundary condition (assumption (ii)). Indeed, it may not even be apparent to the reader at first sight why we call (ii) a Dirichlet boundary condition. Suppose $\phi_n$ is a sequence of functions in $C_c^\infty(D)$, whose support converges  to a subset of the boundary $\partial D$, in the sense that $d_n \to 0 $ (where $d_n$ is defined in (ii)). If $h$ is a Gaussian free field in $D$ (with Dirichlet boundary conditions), we may be tempted to believe that $(h, \phi_n) \to 0$. Unfortunately, without any additional assumption this is not necessarily the case, even if $\|\phi_n\|_1 $ is bounded (to see why, consider the uniform distribution in a ball of radius $\eps$ at distance $\eps$ from the boundary). Instead, in order for $(h, \phi_n)$ to converge to zero we need an extra condition which guarantees that the mass of $\phi_n$ is sufficiently ``spread out''. There are several different ways that such a condition could be formulated. In \cite{BPR18} we assumed that for $D= \D$, $(h, \phi_n)\to 0 $ for sequences $\phi_n $ which are bounded in $L^1$ and \emph{rotationally symmetric}. However, in the present article, we will need $\phi_n$ to be asymptotically supported on a \emph{proper} subset of the boundary (see the definition of $p_u$ in \eqref{eq:sine}) and so rotational invariance of the support of $\phi_n$ is not sufficient. Instead we assume that $\phi_n$ converges to 0 in $H^{-1}(D)$. This turns out to be the most convenient meaning of ``sufficiently spread out'' in the present setting.	\medskip

Before stating our results, we recall the definition of a Gaussian free field (with Dirichlet boundary conditions) on a domain $D \subset \C$.

\begin{defn} \label{def::gff} A mean zero Gaussian free field $h_{\gff} =h^D_{\gff}  $
	with zero boundary conditions is a stochastic process indexed by test functions $(h_{\gff}, \phi)_{\phi \in C_c^\infty(D)}$ such that:
	
	\begin{itemize}
		\item $h_{\gff}$ is a centered Gaussian field; for any $n\ge 1$ and any set of test functions $\phi_1,\cdots, \phi_n \in C_c^\infty(D)$, $((h_{\gff},\phi_1),\cdots, (h_{\gff},\phi_n))$ is a Gaussian random vector with mean ${\mathbf{0}}$;
		\item for any two test functions $\phi_1,\phi_2 \in C_c^\infty(D)$,
		$$
		\E[(h_{\gff},\phi_1) , (h_{\gff},\phi_2)] = \int_{D} G^D(z,w) \phi_1(z)\phi_2(w)dzdw
		$$
		where $G^D$ is the Green's function with Dirichlet boundary conditions on $D$.
	\end{itemize}
\end{defn}

The main technical content of this paper is summarised by the following proposition, whose most important aspect states that moments of order $\xi$ as in Assumptions \ref{ass:ci_dmp}, together with domain Markov property and conformal invariance, imply a moment of order 4.

\begin{prop}\label{prop:fourth_moment}
Assume that $(\Gamma^D)_D$ satisfies Assumptions \ref{ass:ci_dmp}. Then in fact:
\begin{enumerate}[(1)]
\item $\E[(h^D,\phi)^4]<\infty $
for every $\phi\in C_c^\infty(D)$;
\item the bilinear form $K_2^D$ on $C_c^\infty(D)\times C_c^\infty(D)$  defined by
\[ \E[(h^D,\phi)(h^D,\phi')]=K_2^D(\phi,\phi'), \quad \quad \phi, \phi' \in C_c^\infty(D)\]
is continuous; and
\item the convergence in (ii) of Assumptions \ref{ass:ci_dmp} also holds in $L^2$.
\end{enumerate}
\end{prop}

As a direct consequence we obtain the following theorem, which is the main result of this paper.

\begin{thm}\label{thm::characterisation_gff}
	Suppose the collection of laws $\{\Gamma^D\}_{D\subset \mathbb{C}}$ satisfy Assumptions \ref{ass:ci_dmp} and let $h^D$ be a sample from $\Gamma^D$. Then there exists $\sigma\ge 0$ such that $h^D = \sigma h_{\gff}^D$ in law, as stochastic processes.
\end{thm}

\begin{proof}
This is a direct consequence of Proposition \ref{prop:fourth_moment} and \cite[Theorem 1.6]{BPR18}.
\end{proof}

\paragraph{Proof idea:} In order to explain the new ideas required for Theorem \ref{thm::characterisation_gff}, it is helpful to first recall the main steps in the proof of \cite[Theorem 1.6]{BPR18}. \medskip

\emph{Sketch of proof of \cite[Theorem 1.6]{BPR18}.} The proof of Theorem 1.6 in \cite{BPR18} can be broken into two distinct parts: (1) showing that the field is Gaussian (i.e., that $h^D$ is a Gaussian process for each $D$) and (2) showing that it has the correct covariance structure.
In fact, once Gaussianity is known, proving (2) is rather straightforward. It boils down to the fact that the Greens' function is characterised by harmonicity away from the diagonal and logarithmic blow-up along the diagonal -- see \cite{BPR18}.

Proving (1) is rather more challenging. The key step in \cite{BPR18} is to show that ``circle averages'' around points are jointly Gaussian. That is, for any finite set of points, the joint law of the circle averages is Gaussian. The circle average process of a Gaussian free field $h^D$ around a point $z\in D$ is, roughly speaking, the process $(h,
\phi_t)_{t\ge 0}$, where $\phi_t$ is uniform measure on the circle of radius $\e^{-t}$ around $z$. More precision is required for a rigorous definition, since the $\phi_t$ are not smooth test functions, but this can be dealt with by approximating the $\phi_t$ appropriately. Once it is known that circle averages are jointly Gaussian, it is easy to deduce (1), because the field can be approximated by circle averages with small radii, and limits of Gaussians are Gaussian.

To address the question of showing Gaussianity of circle averages, let us consider the case where $D=\D$ is the unit disc, and we take averages around a single point: the origin. It is well known and easy to see that for a GFF in $\D$, the circle average process around $z = 0$ is a constant multiple of Brownian motion. For our given process $h^\D$, the domain Markov property together with scale invariance (a special case of conformal invariance) shows that the circle average process has independent and stationary increments. However, one cannot immediately deduce that it is Brownian motion, which would of course yield Gaussianity. More work is required to eliminate processes with jumps (e.g.\ compound Poisson processes, symmetric stable processes etc.). In \cite{BPR18}, a fourth moment assumption on the field was used to apply Kolmogorov's criterion, and thereby prove that the circle average process possesses an almost surely continuous modification. This modification must then be Brownian motion and, in particular, Gaussian. In fact, we can generalise this argument to show that arbitrary linear combinations of circle averages around multiple points must also be Gaussian, which completes the key step of the proof.

\emph{Sketch of proof of Proposition \ref{prop:fourth_moment}.}
The major challenge in this article is to reach the same conclusion \emph{without} the fourth moment assumption. In contrast to the above approach, we will simply aim to prove Gaussianity of single circle averages, rather than linear combinations of averages around multiple points.  Note that this does not immediately imply \emph{joint} Gaussianity of circle averages (for which significantly more work would be needed).
However, it is enough (with a little extra work) to prove existence of fourth moments (\cref{prop:fourth_moment}) and given the result of \cite{BPR18}, this concludes the proof of \cref{thm::characterisation_gff}.

To summarise: the main step of the proof \emph{in this article} is to show existence of an a.s.\ continuous modification of the circle average process around $z=0$ for $h^\D$ (the given field in the disk $\D$) assuming only $\xi$th moments of the field for some $\xi>1$. See \cref{cor:circ_avg_cont} and  \cref{prop:circ_av_bm}. Achieving this is not merely a technical upgrade of the idea used in \cite{BPR18}; a new input is required.

Namely, in \eqref{eq:sine} we introduce  a certain \textbf{sine-average process} for the field $h^\H$, on semi-circles in the upper half plane. Its value at a given semi-circle can be viewed as the average of $h^\H$ with respect to a hitting measure for half plane
\textbf{It\^o excursions} from $0$. As a result, one can easily construct a parametrisation (with respect to the semi-circle radius), under which the resulting process satisfies:
\begin{itemize}
	\item (one-dimensional) Brownian scaling; and crucially
	\item a certain \textbf{``harness''} property, as introduced by Hammersley in \cite{harness} (see also \cite{Wes93,Williams_harness}).
\end{itemize} The increments of this process are easily checked to be independent; however, there is no reason \emph{a priori} why they should be stationary. Nonetheless, we are able to formulate a (new) characterisation of Brownian motion in terms of this harness property and use this to show that the sine-average process must be a Brownian motion. This characterisation is given in Proposition \ref{prop:char_BM}, and is an extension of a result proved in \cite{Wes93}. Crucially, our extension does not require as many moments as \cite{Wes93}; in fact moments of \emph{any} order $\xi>0$ suffice.

 From this point, we use rotational invariance and the domain Markov property to ``average out'' the semi-circle sine-averages of $h^\H$ and relate them to circle averages of $h^\D$. The consequence is existence of a continuous modification of the circle-average process around $0$ for $h^\D$. For this last step, one needs to precisely control the behaviour of the harmonic part in a domain Markov decomposition of $h^\D$, which forms the main technical part of the argument. This is where the assumption $\xi >1$ is used. Having done this, the proof of Proposition \ref{prop:fourth_moment} is concluded. \qed

\begin{remark}
 Consider a family of fields $(h^D)_D$ in simply connected domains $D$, that assign values $(h^D, \phi)$ to smooth test functions $\phi$. Theorem \ref{thm::characterisation_gff} shows that conformal invariance and the domain Markov property (in the sense of Assumptions \ref{ass:ci_dmp}) are incompatible with these $(h^D,\phi)$s having $\alpha$-stable (rather than Gaussian) distributions, for any value of the index $\alpha \in (1,2)$. Comparing to the better understood one-dimensional situation, a (1d) $\alpha$-stable process has different scaling properties to those of (1d) Brownian motion. Since scaling is a special type of conformal mapping, this suggests that ``natural $\alpha$-stable analogues'' of the GFF cannot enjoy conformal invariance. Our Theorem can be viewed as a rigourous justification of this informal heuristic when $\alpha \in (1,2)$.

We mention here that some variants of higher dimensional stable fields have been defined and studied before, see \cite{kumar1972stable} and also \cite{cipriani2016divisible} for a limiting construction. It will be interesting to find a suitable characterisation theorem for such fields.
\end{remark}

In view of the above remark, it is natural to wonder whether \emph{any} moments assumptions are needed to characterize the GFF.
\begin{question}\label{Q:xi}
What are the minimal moment assumption necessary for  \cref{thm::characterisation_gff} to hold? Do moments of order $\xi$ for any $\xi>0$ suffice?
\end{question}

\paragraph{Acknowledgements} We thank Scott Sheffield and Juhan Aru for some inspiring discussions. Part of this work was carried while all three authors visited Banff on the occasion of the programme ``Dimers, Ising Model, and their Interactions''. We would like to thank the organisers as well as the team in BIRS for this opportunity and their hospitality. Finally, we would like to thank the anonymous referees for many suggestions that helped us to improve the presentation of the paper.

\section{Some elementary results and estimates}
\subsection{Independent random variables}

\begin{lemma}
	\label{lem:indep_moments}
	Suppose that $(X,Y)$ are real-valued random variables defined on the same probability space, and that $X$ and $Y$ are independent. Then for any $\xi>0$,  \[\mathbb{E}[|X+Y|^\xi]<\infty \Rightarrow \mathbb{E}[|X|^\xi]<\infty \text{ and } \mathbb{E}[|Y|^\xi]<\infty. \]
\end{lemma}

\begin{proof}
Fix some $M$ such that $\mathbb{P}(|Y|\le M)\ge 1/2$ and note that $|X/(X+Y)|\mathbf{1}_{\{|Y|\le M, |X|\ge 2M\}}\le 2$ (it is less than 1 if $X$ and $Y$ have the same sign, and less than $2$ otherwise). Then $\mathbb{E}[|X|^\xi \mathbf{1}_{\{|X|\le 2M\}}]\le (2M)^\xi$ and
\[ \mathbb{E}\left[|X|^\xi \mathbf{1}_{\{|X|\ge 2M\}}\right]\le 2 \mathbb{E}\left[\left|\frac{X}{X+Y}\right|^\xi \, \left|X+Y\right|^\xi \mathbf{1}_{\{|Y|\le M, |X|\ge 2M\}}\right] \le 2^{1+\xi} \mathbb{E}\left[\left|X+Y\right|^\xi \right]<\infty.\]

Symmetrically, $\mathbb{E}[|Y|^\xi]<\infty$.
\end{proof}

\begin{lemma}[Von Bahr--Esseen \cite{vBE65}]\label{lem:vbe} Let $r \ge 1$.
\begin{enumerate}[(i)]
	\item Suppose that $X,Y$ are random variables with $\mathbb{E}[|X|^r]<\infty, \mathbb{E}[|Y|^r]<\infty, \mathbb{E}[Y|X]=0$ a.s.\ Then
	$\mathbb{E}[|X+Y|^r]\ge \mathbb{E}[|X|^r]$.
	\item Suppose  in addition that $r\le 2$ and that  $(X_1,\cdots, X_n)$ are independent, centred random variables with $\mathbb{E}[|X_j|^r]<\infty$ for $1\le j \le n$. Then
	$\mathbb{E}[|\sum_{j=1}^n X_j |^r]\le 2 \sum_{j=1}^n \mathbb{E}[|X_j|^r]$.
\end{enumerate}	
	
\end{lemma}
\subsection{Immediate consequences of the domain Markov property}

\begin{lemma}\label{lem:unicity_decomposition}
The assumption of zero boundary conditions implies that the domain Markov decomposition from (iv) is unique.
\end{lemma}
\begin{proof} This is very similar to the proof of \cite[Lemma 1.4]{BPR18}, but we include it since some arguments are slightly different.
	
Suppose that we have two such decompositions:
\begin{equation}\label{eqn::DMP_uniqueness}
h^D = h^{D'}_D+\ph_D^{D'} =  \tilde{h}^{D'}_D+\tilde{\ph}_D^{D'}.
\end{equation}
Pick any $z\in D'$ and let $f:D'\to \D$ be a conformal map that sends $z$ to $0$. Further, let $(\phi_n)_{n\ge 1}$ be a sequence of nonnegative radially symmetric, mass one functions in $C_c^\infty(\D)$, that are eventually supported outside any $K\Subset \D$. It is easy to check that $\phi_n\to 0$ in $H^{-1}(\D)$ as $n\to \infty$, and if we set $\tilde{\phi}_n := |f'|^2 (\phi_n \circ f)$ for each $n$, then (as discussed below \cref{ass:ci_dmp})  $\tilde{\phi}_n$ converges to $0$ in $H^{-1}(D')$ as well. Hence, the assumption of Dirichlet boundary condition implies that $(h^{D'}_D-\tilde{h}^{D'}_D,\tilde{\phi}_n )\to 0$ in probability as $n\to \infty$. In turn, by (\ref{eqn::DMP_uniqueness}), this means that $(\ph_D^{D'}-\tilde{\ph}_D^{D'},\tilde{\phi}_n) \to 0$ in probability.

However, since $(\ph_D^{D'} - \tilde{\ph}_D^{D'})$ restricted to $D'$ is a.s. equal to a harmonic function, and since the $\phi_n$'s are radially symmetric with mass one, we have that
	\[ (\ph_D^{D'}-\tilde{\ph}_D^{D'},\tilde{\phi}_n) =((\ph_D^{D'}-\tilde{\ph}_D^{D'})\circ f^{-1}, \phi_n)=(\ph_D^{D'}-\tilde{\ph}_D^{D'})\circ f^{-1}(0)=\ph_D^{D'}(z)-\tilde{\ph}_D^{D'}(z)\] for every $n$. This implies that for each fixed $z\in D'$, $\ph_D^{D'}(z)=\tilde{\ph}_D^{D'}(z)$ a.s. Applying this to a countable dense subset of $z\in D'$, together with the fact that $(h^D,\phi)=(\ph_D^{D'},\phi)=(\tilde{\ph}_D^{D'},\phi)$ a.s. for any $\phi$ supported outside of $D'$, then implies that $\ph_D^{D'}$ and $\tilde \ph_D^{D'}$ are a.s. equal as stochastic processes indexed by $C_c^\infty(D)$.
\end{proof}

	Now, suppose that $D''\subset D'\subset D$ and $h^D$ is a sample from $\Gamma^D$. Applying the domain Markov property to $h^D$ in $D'$ and $D''$ respectively, we can write $h^D=h_D^{D'}+\varphi_D^{D'} \text{ and } h^D=h_D^{D''}+\varphi_{D}^{D''}.$ We can further decompose $h_D^{D'}=h_{D'}^{D''}+\varphi_{D'}^{D''}$ by applying the domain Markov property to $h_D^{D'}$ in $D''$.
\begin{lemma}\label{lem:nested_dmp}	
	As stochastic processes indexed by $C_c^\infty(D)$, we have that $h_D^{D''}=h_{D'}^{D''}$ and $\varphi_D^{D''}=\varphi_D^{D'}+\varphi_{D'}^{D''}$ a.s. (where the latter is an independent decomposition).
\end{lemma}
\begin{proof}
	This follows by writing $h^D=h_D^{D''}+\varphi_D^{D''}$ and $h^D=h_D^{D'}+\varphi_{D}^{D'}=h_{D'}^{D''}+\varphi_{D'}^{D''}+\varphi_{D}^{D'}$ and applying Lemma \ref{lem:unicity_decomposition}. \end{proof}

\begin{lemma}\label{lem:harm_ci}
		Suppose $D$ is simply connected and that $D'\subset D$ is a simply connected Jordan domain. Then if $h^D=h^{D'}_D+\varphi_D^{D'}$ is the domain Markov decomposition of $h^D$ in $D'$ and $f:D\to f(D)$ is conformal, with $f(D')\subset f(D)$ a Jordan domain and $h^{f(D)}=h_{f(D)}^{f(D')}+\varphi_{f(D)}^{f(D')}$, we have that
		\[\varphi_D^{D'}=\varphi^{f(D')}_{f(D)}\circ f \text{ in law}\]
		as harmonic functions in $D'$.
		
	\end{lemma}
\begin{proof}
For $\phi\in C_c^\infty(D')$ let us denote $\phi^f(z) = |(f^{-1})'|^2 \phi \circ f^{-1}(z) $, so that $\phi^f\in C_c^\infty(f(D'))$. Then by conformal invariance (\cref{ass:ci_dmp}(iii)) it follows that
	\begin{equation*}
	(h^{D} , \phi ) \overset{(d)}{=} (h^{f(D)}, \phi^f) \text{ and } (h^{D'} , \phi ) \overset{(d)}{=} (h^{f(D')}, \phi^f).
	\end{equation*}
By uniqueness of the domain Markov decomposition (\cref{lem:unicity_decomposition}), it then follows that
	$$
	(\varphi_{D}^{D'} , \phi) \overset{(d)}{=} (\varphi_{f(D)}^{f(D')}, \phi^f)
	$$
	and since $\varphi$ is harmonic, this is exactly the statement that
	$$
	\int_{D'}  \varphi_{D}^{D'} (z)\phi(z)dz \overset{(d)}{=} \int_{f(D')} \varphi_{f(D)}^{f(D')}(z)\phi^f(z)dz = \int_{D'} \varphi_{f(D)}^{f(D')} (f(w))\phi(w)dw,
	$$
	where the last equality is just the change of variables formula. Since this holds for all $\phi \in C_c^\infty(D')$, this completes the proof. 	
\end{proof}

\subsection{A priori moment bounds}

We are going to give some bounds on the moments of harmonic functions arising from the domain Markov property. Note that if $z\in D'\subset D$ and $\varphi_D^{D'}$ is such a function, then by harmonicity we can write $\varphi_D^{D'}(z)=(\varphi_D^{D'},\phi)=(h^D,\phi)-(h_D^{D'},\phi)$ for some properly chosen $\phi\in C_c^\infty(D')\subset C_c^\infty(D)$ (e.g., take $\phi$ to be a spherically symmetric bump function which integrates to 1). Therefore \[\mathbb{E}[|\varphi_D^{D'}(z)|^p]<\infty\] for all $0\le p\le \xi$. Moreover, if $D''\subset D'$, then by \cref{lem:nested_dmp} and \cref{lem:vbe}(i), we have
\begin{equation}\label{eq:mono_moments}
\mathbb{E}[|\varphi_D^{D'}(z)|^p] \le \mathbb{E}[|\varphi_D^{D''}(z)|^p]
\end{equation}
for all $p\in [1,\xi]$. (Note that $\E(\varphi_{D'}^{D''}(z)|\varphi_D^{D'}(z))=0$, since $\varphi_{D'}^{D''}$ and $\varphi_{D}^{D'}$ are independent and $\E(\varphi_{D'}^{D''}(z))=\E((h^D,\phi)-(h_D^{D'},\phi))=0$ by assumption. Thus we are justified in applying \cref{lem:vbe}(i).)

\begin{lemma}\label{lem:moment_bound}
	Suppose that $D'\subset D$ and that $z\in D'$. Then there exists a universal constant $C$ (i.e., not depending on $z,D,D'$) such that
	for all $p\in [0,\xi\wedge 2]$
	\[ \mathbb{E}[|\varphi_D^{D'}(z)|^p]\le C \left(\log\left(\frac{d(z,\partial D)}{d(z,\partial D')}\right)\vee 1\right)
	\]
\end{lemma}

\begin{proof}
Let $r:=d(z,\partial D')/2$ and $R:=d(z,\partial D)/2$. By Jensen's inequality we need only consider the case $p=\xi$. In this case, since $\xi>1$ and $B_z(r)\subset D'$, we may further assume by \eqref{eq:mono_moments} that $D'=B_z(r)$.
	
Now we iteratively apply \cref{lem:nested_dmp}. Let $B_k=B_z(2^k r)$ for $k\in \mathbb{N}_0$, and let $N:=\sup_{k\in \mathbb{N}_0} B_k\subset D$ so that $ N\le \log(R/r)/\log(2)$.
Then we may write
\[ \varphi_{D}^{D'}(z)=\varphi_D^{B_N}(z)+\sum_{k=0}^{N-1} \varphi_k(z)\]
where the $\varphi_k(z)$ are independent and, by conformal invariance, each distributed as $\varphi_{\D}^{\D/2}(0)$. Therefore by \cref{lem:vbe}(ii), it follows that
\[ \mathbb{E}[|\varphi_D^{D'}(z)|^\xi]\le 2( \mathbb{E}[|\varphi_D^{B_N}(z)|^\xi] + N \mathbb{E}[|\varphi_{\D}^{\D/2}(0)|^\xi]).\]
Now $\mathbb{E}[|\varphi_{\D}^{\D/2}(0)|^\xi]$ is bounded by some universal constant. Moreover, so is $\mathbb{E}[|\varphi_D^{B_N}(z)|^\xi]$: if $f:D\to \D$ maps $z$ to $0$, then $f(B_N)\supset (1/32)\D$ by the Koebe quarter theorem, and it therefore follows from conformal invariance and \eqref{eq:mono_moments}  that $\mathbb{E}[|\varphi_D^{B_N}(z)|^\xi]\le \mathbb{E}[|\varphi_{\D}^{(1/32)\D}(0)|^\xi]$). This completes the proof.
\end{proof}

\section{Sine-averages and harmonic functions}\label{sec:sine_avgs}
In the following we will denote the upper unit semi disc $\D\cap \H$ by $\D^+$.  For $r>0$, we denote by $r\D^+$ the scaled semi disc $\{z\in \H: |z|<r\}$, and for compactness, write \[\mathsf D_u:=\frac{1}{\sqrt{u}}\D^+; \text{ for } u>0.\]

For $u>0$, we define $p_u$ to be the measure that integrates against $\phi\in C_c(\C)$ as
\begin{equation}\label{eq:sine}
(\phi,p_u)= p_u(\phi):=\sqrt{u} \int_{0}^{\pi}  \sin(\theta) \phi\left(\frac{\e^{i\theta}}{\sqrt{u}}\right) \, d\theta.
\end{equation}

Note that $p_u$ is supported on the circle of radius $r_u = 1/ \sqrt{u}$ and that its total mass is $2/r_u = 2 \sqrt{u}$.

The motivation for defining these measures comes from the fact that $\mathfrak{h}(re^{i\theta}) = \frac1r\sin(\theta)$ is harmonic in the upper half plane with zero boundary conditions (except at the origin). In fact, $\mathfrak{h}$ can be interpreted as the hitting density on a circle of radius $r$, for an It\^o excursion in the upper half plane starting from zero. While our proofs can be written without referring to this interpretation, it may be useful for the intuition nonetheless, so we will now explain how to state this more precisely.

We start by recalling some background about such excursions (see Chapter 5.2 in \cite{Lawler-book} for further details). Let $\P_{i \eps}$ denote the law of Brownian motion starting from $i\eps$, killed when it leaves the upper half plane $\H$. By definition, the \textbf{It\^o excursion measure} from zero is the (infinite) measure $\mathcal{N}$ obtained as the vague limit
$$
\mathcal{N} : = \lim_{\eps \to 0} \frac1{\eps} \P_{i \eps}
$$
which is supported on continuous trajectories  $\omega $ starting from zero, such that $\omega (t) \in \H$ for $t \in (0, \zeta)$ where $\zeta = \zeta(\omega)$ is the lifetime of the excursion, and such that $\omega(t) = \omega(\zeta) \in \R$ for any $t \ge \zeta$. A ``sample'' from $\mathcal{N}$ will later be called a half plane excursion. More generally, the corresponding excursion measure can be defined on any simply connected domain $D$ from an analytic boundary point $z \in \partial D$ (meaning that there is a conformal map $f:D\to \H$ mapping $z$ to $0$ that extends analytically to a neighbourhood of $z$ on $\partial D$) and we then denote it by $\mathcal{N}_{z, D}$. These measures are conformally covariant, in the sense that for a conformal map $f: D \to \H$ as above, the image of $\mathcal{N}_{z, D}$ under $f$ is given by $|f'(z)| \mathcal{N}_{0, \H }$ \cite[p126]{Lawler-book}.

Note that even though $\mathcal{N}$ has infinite mass we can easily make sense of conditional laws $\mathcal{N}(\cdot | E)$ when $\mathcal{N}(E)\, \in(0,\infty)$, thus resulting in probability measures. We record the following lemma.
\begin{lemma}\label{L:exc}
The total mass of half plane excursions reaching $\partial (r \D) \cap \H$ is $4/(\pi r)$. In fact, the mass of excursions leaving $r\D \cap \H$ through the arc $(re^{ia}, re^{ib})$ is precisely
$$
\frac2{\pi r} \int_a^b \sin(\theta) d\theta
$$
for any $0 \le a \le b \le \pi$.
\end{lemma}
\begin{proof}
Note that when $D = \H$ and $z = \infty$, the measure $\mathcal{N}_{\infty, \H} ( X(\zeta_\H) \in [a, b]) =  (b-a)/\pi$ on $\R$, is nothing but Lebesgue measure (here $\zeta_D$ denotes the first time that the excursion $X$ leaves the domain $D$, i.e., its lifetime). This is easy to check, as starting from a point $ir$ (with $r>0$) the hitting distribution of $\R$ by a Brownian motion has the Cauchy distribution scaled by $r$, which tends to $\pi^{-1}$ times Lebesgue measure on $\R$ as $r \to \infty$.

For $r >0$, consider the conformal maps
$$f(z) = z + \frac{r^2}{z} = r( \frac{r}{z} + \frac{z}{ r}),$$
that map $\H \setminus (r\D)$ to $\H$ and satisfy $f(\infty) = \infty$ with $|f'(\infty) | =1$. Note that $f(r e^{i \theta}) = 2 r \cos (\theta)$. In particular $f$ sends the semicircle of radius $r$ to the interval $[-2r, 2r]$, of length $4r$. Hence if $\tau_r$ is the first hitting time of this circle, we have
$$
\mathcal{N}_{\infty, \H} ( \tau_{r} < \zeta) = 4r/\pi.
$$
The first claim of the lemma follows from this after applying the inversion map $z \mapsto -1/z$ (which sends $\infty$ to 0, leaves $\H$ invariant, and transforms $r\D$ into $(1/r) \D$). The second claim follows easily after noting that the derivative in $\theta$ of $f(re^{i \theta})$ is $-2 r\sin (\theta)$.
\end{proof}
\begin{remark}\label{R:exc}
For later reference, it may be useful to note that half plane excursions enjoy the following Markov property: conditionally upon hitting the circle of radius $r$, the law of an excursion after this time is simply that of Brownian motion killed upon leaving $\H$.\end{remark}

Combined with the domain Markov property and scale invariance of our fields, the upshot is that when we ``integrate $h^\H$ against $\mathfrak{h}$ on the semi-circle of radius $1/\sqrt{u}$ around $0$'' - equivalently ``test $h^\H$ against $p_u$'' - and view this as a process in $u$, it will satisfy both Brownian scaling and a certain Markovian property (note that $u = 0$ corresponds to testing $h^\H$ near the point at $\infty$). As a consequence, we may deduce that the process is Brownian motion -- see \cref{sec:char_BM}.
However, the reader may recall from the introduction that we really want \emph{circle averages}, say for $h^\D$, to be Brownian motions. Since these processes are easily shown to have independent and stationary increments, this would be immediate if we knew that \emph{they} satisfied Brownian scaling. Unfortunately, this seems very hard to deduce directly from \cref{ass:ci_dmp}. So, we introduce the measures $p_u$ (and associated sine-averages for $h^\H$, see below) instead, and will later relate them to circle averages in \cref{sec:circ_avg_gaussian}. We remark that alternative measures to $p_u$, for example correctly defined variants in cones, could play the same role. The current set-up has been chosen as it seems to be the neatest.

Now, in order to make sense of ``testing $h^\H$ against $p_u$'' we need to first approximate $p_u$ by some smooth test functions.
For $\delta\in (0,\pi/2)$ we let $p_u^\delta$ be defined in the same way as $p_u$, but replacing $\sin(\theta)$ in the integral above with $\sin(\theta) \chi^\delta(\theta)$, where $\chi^\delta:[0,\pi]\to [0,1]$ is smooth, equal to $1$ in $[\delta, \pi-\delta]$, and equal to $0$ in $[0,\delta/2]\cup [\pi-\delta/2,\pi]$.
Finally, for $\eta:[0,1]\to [0,1]$ a smooth bump function with $\int_{0}^{1}\eta(y) \, dy=1$, we define $\eta^{\delta}(\cdot):=\frac{1}{\delta}\eta(\frac{\cdot}{\delta})$ and denote by $p_u^{\delta,in},p_u^{\delta,out}$ the measures that integrate against $\phi\in C_c(\C)$ as
\[ (\phi,p_u^{\delta,in\,}):= \int_0^\delta (\phi,p^\delta_{u(1+x)})\, \eta^{\delta}(x) \, dx \;\; ; \;\; (p_u^{\delta,out},\phi):=\int_0^\delta (\phi,p^\delta_{u(1-x)})\, \eta^{\delta}(x) \, dx.\]
Thus $p_u^{\delta,in},p_u^{\delta,out}$ are smooth ``fattenings'' of the measure $p_u$ to the inside and outside of the arc $\partial (\frac1{\sqrt{u} } \D^+)$ respectively, that are also ``cut off'' away from the real line (so as to have compact support in $\H$). The reason for these definitions is the following:

\begin{remark}
	\label{rmk:fattening_smooth}
We have that for some $p_u^{\delta,in/out}\in C_c^\infty(\C)$ (note the abuse of notation $p_u^{\delta, in/out}$ for both measure and density here): $$(p_u^{\delta,in/out},\phi)=\int_{\C} p_u^{\delta,in/out}(z)\phi(z) \, dz.$$ We remark that it is possible to write down an explicit expression for $p_u^{\delta,in/out}(z)$, but we do not need it.
\end{remark}
The upshot is that we can define $$(h^D,p_u^{\delta,in/out})$$ (where $p_u^{\delta,in/out}$ refers to the smooth density) for any $D$ such that $\mathrm{Support}(p_u^{\delta,in/out})\Subset D$ (e.g., $D=\D^+$ or $D=\H$).

 \begin{lemma}\label{lem:harmonic_sines}
\begin{enumerate}[(a)]	\item Suppose that $u>0$ and $\varphi$ is a harmonic function in $\sou$, that can be extended continuously to a function on $\sou\cup(-\frac{1}{\sqrt{u}}, \frac{1}{\sqrt{u}})$
	 that is equal to zero on $(-\frac{1}{\sqrt{u}}, \frac{1}{\sqrt{u}})$. Then $(\varphi,p_r)_{r\in (u,\infty)}$ is constant.

	\item Suppose that $u>0$ and $\varphi$ is a harmonic function in $\H\setminus \overline{\sou}$ that can be extended continuously to $0$ on $(-\infty,- \frac{1}{\sqrt{u}})\cup (\frac{1}{\sqrt{u}},\infty)$. Then $(\varphi,p_s)_{s\in (0,u)}$ is a linear function of $s$.
	\item Suppose that $0<s<r<\infty$ and $\varphi$ is a harmonic function in $\sos\setminus \overline{\sor}$ that can be extended continuously to 0 on $(-\frac{1}{\sqrt{s}},-\frac{1}{\sqrt{r}})\cup (\frac{1}{\sqrt{r}},\frac{1}{\sqrt{s}})$. Then $(\varphi,p_u)_{u\in (s,r)}$ is a linear function of $u$.
\end{enumerate}
\end{lemma}

\begin{remark}
We observe that (a) is easily seen from the perspective of It\^{o} excursions. By \cref{L:exc}, we can represent $(\varphi,p_r)$ for any $r>u$ by $\frac{\pi}{2}\mathcal{N}_{0,\H}(\varphi(X_{\tau_{(1/\sqrt{r})}\wedge \zeta}))$
 where $\tau_{(1/\sqrt{r})}$ is the first hitting time of the semicircle of radius $(1/\sqrt{r})$ centred at $0$. For $s\ge r$, since $\varphi$ is assumed to be 0 on $(-1/\sqrt{u},1/\sqrt{u})$, we can apply the Markov property, \cref{R:exc}, of the excursion $X$  at $\tau_{(1/\sqrt{s})}\wedge \zeta$. This gives $(\varphi,p_r)=\sqrt{s} \int_0^\pi \sin(\theta) \E_{\frac{e^{i\theta}}{\sqrt{s}}}[\varphi(B_{\tau_{\partial\sor}})] \, d\theta$ for $B$ a complex Brownian motion. By harmonicity of $\varphi$, this quantity is equal to  $(\varphi,p_s)$ as required.

Actually, it can be seen from the argument above that the constant value of $(\varphi,p_r)$ for $r>u$, is equal to $\pi/2$ times the normal derivative, directed \emph{into} $\H$, of $\varphi$ at the origin. Indeed, we saw that for any such $r$, $$(\varphi,p_r)=\frac{\pi}{2}\mathcal{N}_{0,\H}(\varphi(X_{\tau_{(1/\sqrt{r})}\wedge \zeta}))=\frac{\pi}{2}\lim_{\eps\to 0}\eps^{-1}\mathbb{E}_{i\eps}(\varphi(B_{\tau_{(1/\sqrt{r})}\wedge \zeta}))=\frac{\pi}{2}\lim_{\eps\to 0} \eps^{-1}\varphi(i\eps),$$
 where the second equality is by definition of $\mathcal{N}_{0,\H}$ and the third is by harmonicity of $\varphi$.
\end{remark}

{Since it is simpler for (b) and (c), the full proof of \cref{lem:harmonic_sines} below is of a more deterministic nature.}
\begin{proof}
 	
	Write $\varphi(r\e^{i\theta})=\varphi(r,\theta)$ and  $f(u)=(\varphi,p_u)=\sqrt{u} \int_0^\pi \sin(\theta)\varphi(1/\sqrt{u},\theta) \, d\theta$. We will show that $f''\equiv 0$ on $(s,r)$, which implies (c). This in turn implies (b), by taking $s$ to $0$.
	
Take any $u\in (s,r)$. Let us first remark, in order to justify differentiation under the integral and integration by parts in what follows, that $\varphi$ is in fact very regular in open neighbourhoods of $\pm (1/\sqrt{u})$ inside $\sos\setminus \overline{\sor}$. Indeed since $\varphi$ extends continuously to $0$ on neighbourhoods of $\pm (1/\sqrt{u})$ in $\R$, it can be extended by Schwarz reflection to a harmonic function in open balls $B_{\pm 1/\sqrt{u}}(\eps)\subset \C$ for some $\eps$. See, for example, \cite[\S 7.5.2]{krantz}.
In particular $\frac{\partial \varphi }{\partial \theta} $ remains bounded in neighbourhoods of $\pm 1/\sqrt{u}$.
Now we compute
\begin{eqnarray*} \frac{d^2}{du^2}(\sqrt{u} \varphi(1/\sqrt{u},\theta)) & = & \frac{1}{4u^{5/2}}\left( \frac{\partial^2}{\partial r^2}\varphi(1/\sqrt{u},\theta)+ \sqrt{u}\frac{\partial}{\partial r}\varphi(1/\sqrt{u},\theta)-u\varphi(1/\sqrt{u},\theta)  \right) \\ & = & -\frac{1}{4u^{3/2}} \left(\frac{\partial^2}{\partial \theta^2}\varphi(1/\sqrt{u},\theta)+\varphi(1/\sqrt{u},\theta)\right)
	 ,\end{eqnarray*}
 using harmonicity of $\varphi$ for the final identity. Differentiating under the integral in the expression for $f(u)$, and apply integration by parts twice with respect to $\theta$, we see that $f''(u)=0$.
\end{proof}

\begin{prop}\label{prop:def_sine_avg}
	Let $h^\H$ be a sample from $\Gamma^\H$. Then for any $u\in (0,\infty)$ the limits
	\begin{equation} \label{eq:hpu_def} \lim_{\delta\downarrow 0}(h^\H,p_u^{\delta,in}) \text{ and }
	\lim_{\delta\downarrow 0} (h^\H,p_u^{\delta,out}) \end{equation}
	exist in probability and in $L^1$,
	and are equal a.s.
	We define this limiting quantity to be the $(1/\sqrt{u})$-\textbf{sine average} of $h^\H$, and denote it (with a slight abuse of notation) by $(h^\H, p_u)$. Recall the notation $h^\H=h_\H^D+\varphi_\H^D$ for the domain Markov decomposition of $h^\H$ in $D\subset \H$.  We also have that with probability one:
	\begin{equation}\label{eq:hpu_alt_def} (h^\H,p_u)=(\varphi_{\H}^{\sou},p_r)  \text{ for all }  r> u  \, \text{ \bfseries{and} } \,  (h^\H,p_u)=\frac{u}{s}(\varphi_{\H}^{\H\setminus \overline{\sou}},p_s) \text{ for all } s< u. \end{equation}
\end{prop}

\begin{remark}\label{rmk:hpu_process}
This directly implies that for any finite collection $u_1,\cdots, u_n\in (0,\infty)$, the limits in \eqref{eq:hpu_def} hold jointly in probability, and \eqref{eq:hpu_alt_def} holds jointly almost surely.
 In particular, this  defines a consistent family of  finite dimensional marginals, from which we may define the stochastic process $$(h^\H,p_u)_{u\in (0,\infty)}.$$
\end{remark}

Before we begin the proof of \cref{prop:def_sine_avg}, we need the following lemma. It says (albeit in a more specific setting) that if we apply the domain Markov property to our field in a subdomain that shares a section of boundary with the original domain, then the harmonic function can be extended continuously to $0$ on the common section of boundary. This should seem very intuitive, but the proof is a little trickier than one might guess (see for example Fatou's theorem \cite{Fatou} for the kind of conditions that guarantee existence of non-tangential limits for harmonic functions at the boundary).

\begin{lemma}\label{lem:harm_0_boundary}
Suppose that $h^\H=h_\H^{\D^+}+\varphi_{\H}^{\D^+}$ is the domain Markov decomposition of $h^\H$ in $\D^+$. Then $\varphi_{\H}^{\D^+}$ can almost surely be extended continuously to $0$ on $(-1,1)$.
\end{lemma}

\begin{proof}
	 We first show that for any $y\in (-1,1)$:
\begin{equation}\label{eq:bc_harm}
\varphi_{\H}^{\D^+}(y+i\delta)\to 0 \text { in distribution (so also in probability) as } \delta\to 0.
\end{equation}

Without loss of generality, the other cases being very similar, let us assume that $y=0$. Observe that by \cref{lem:harm_ci}  and harmonicity we have that $$\varphi_{\H}^{\D^+}(i\delta)\overset{(d)}{=}\varphi_{\H}^{(1/\delta)\D^+}(i)=(\varphi_{\H}^{(1/\delta)\D^+},\psi),$$ where $\psi\in C_c^\infty(\C)$ is non-negative with $\int_{\C} \psi=1$, supported in $B_i(1/2)$ and rotationally symmetric about $i$. Moreover, by definition of the domain Markov decomposition and conformal invariance, we have that $$(h^\H, \psi)\overset{(d)}{=} (h^{(1/\delta)\D^+},\psi)+(\varphi_{\H}^{(1/\delta)\D^+},\psi) \text{ with } h^{(1/\delta)\D^+},\, \varphi_{\H}^{(1/\delta)\D^+} \text{ independent.} $$ On the other hand, it is easy to see by conformal invariance of $h$ that $(h^{(1/\delta)\D^+},\psi)$ converges in distribution to $(h^\H, \psi)$ as $\delta\to 0$. This implies  that $$(\varphi_{\H}^{(1/\delta)\D^+},\psi)\to 0$$ in distribution and probability as $\delta\to 0$, by standard arguments (for example, considering characteristic functions).

This completes the proof of \eqref{eq:bc_harm}.
We immediately observe that the sequence in \eqref{eq:bc_harm} is uniformly integrable by \cref{lem:moment_bound}, and so \eqref{eq:bc_harm} can be strengthened to say that \begin{equation}\label{eq:bc_harm_e}
\mathbb{E}[|\ph_\H^{\D^+}(y+i\delta)|]\to 0 \text { as } \delta\to 0
\end{equation}
With \eqref{eq:bc_harm_e} in hand, let us now take $I=[a,b]\subset (-1,1)$ arbitrary: we will show that $\varphi_{\H}^{\D^+}$ can almost surely be continuously extended to $0$ on $I$. We denote $\ph=\varphi_{\H}^{\D^+}$ from now on, and fix $J$ such that $I\subset J\subsetneq [-1,1]$.

First, observe that by dominated convergence and \cref{lem:moment_bound}, \eqref{eq:bc_harm_e} implies that $\mathbb{E}[\int_J |\ph(y+i\delta)| \, dy]\to 0$ as $\delta\to 0$ and hence that for some sequence $\delta_k\to 0$, $a_k:=\int_J |\ph(y+i\delta_k)| \, dy$ converges to $0$ almost surely. We also have by \cref{lem:moment_bound} that if $S_J$ is the semicircle centered on $J$, then $M:=\int_{S_J} |\varphi(z)| \, dz$ is almost surely finite. Finally, by harmonicity of $\varphi$, and by dominating the exit density from  $\H+i\delta$ for Brownian motion started from  $z$ with $\Im(z)\ge 2\delta$ by a Cauchy density, we know that there exists some constant $C$ (deterministic, depending on $I,J$) such for any $z\in \D^+$ that is sufficiently close to $I$, $|\varphi(z)| \le M P(z) +C \Im(z)^{-1} a_k$ for all $k$ large enough, where $P(z)$ is the probability that a Brownian motion started from $z$ hits $S_J$ before $J$. Taking $k\to 0$ gives that $|\varphi(z)|\le MP(z)$ a.s. for all such $z$, and so $\ph$ can almost surely be continuously extended to $0$ on $I$.

\end{proof}

Now we can use \cref{lem:harmonic_sines} to prove \cref{prop:def_sine_avg}.
\begin{proof}[Proof of \cref{prop:def_sine_avg}] Observe that for any $u >0$, $\varphi_\H^{\sou}$ can a.s.\ be extended continuously to $0$ on $(-1/\sqrt{u},1/\sqrt{u})$  by scaling and \cref{lem:harm_0_boundary}.
	Hence by Lemma \ref{lem:harmonic_sines}, on an event of probability one,  \begin{equation}
	\label{eq:vp_const}
(\varphi_\H^{\sou},p_r)=:c	\end{equation} is constant for all $r>u$. This implies (since $\eta^\delta$ has mass one and by definition of $p_u^{\delta, in}$) that with probability one, $$(\varphi_\H^{\sou},p_u^{\delta,in})-c=   \int_0^\delta \left(\varphi_\H^{\sou},p^\delta_{u(1+x)})-(\varphi_\H^{\sou},p_{u(1+x)})\right)\, \eta^{\delta}(x)  \, dx  $$
for all $\delta$ small enough. Noting by \cref{lem:moment_bound} that the right-hand side goes to $0$ in $L^1$ as $\delta\to 0$, we can deduce that
$$(\varphi_\H^{\sou},p_u^{\delta,in})\to c \text{ in probability and in } L^1$$
as $\delta\to 0$.

 Therefore, to show that the first limit in \eqref{eq:hpu_def} exists in probability and in $L^1$, and is equal to $c$ almost surely, we need only show that \[\lim_{\delta\downarrow 0}(h^\H-\varphi_\H^{\sou},p_u^{\delta,in})= \lim_{\delta\downarrow 0}(h_\H^{\sou},p_u^{\delta,in})= 0\] in probability and in $L^1$. However, this follows by applying the zero boundary condition assumption to the field $h_\H^{\sou}$.
	
An almost identical line of reasoning using part (b) of \cref{lem:harmonic_sines} implies that the second limit in \eqref{eq:hpu_def} exists a.s.\ and is equal to the constant value of the second expression in \eqref{eq:hpu_alt_def}. Observe that $$(\varphi_{\H}^{\H\setminus \overline{\sou}},p_s)\to 0$$ in probability and in $L^1$ as $s\to 0$ (for example, by bounding its first moment using \cref{lem:moment_bound}).

Thus all that remains is to show that the two limits in \eqref{eq:hpu_def} (or equivalently in \eqref{eq:hpu_alt_def}) coincide a.s.\ For this, we will prove that
 \begin{equation}\label{eqn:letters} c\overset{(a)}{=}\lim_{\delta \downarrow 0} (\varphi_{\H}^{\mathsf D_{u-\delta}},p_u)\overset{(b)}{=}\lim_{\delta\downarrow 0} (\varphi_{\H}^{\mathsf D_{u-\delta}},p_u^{\sqrt{\frac{u}{u-\delta}}-1, out})\overset{(c)}{=}\lim_{\delta\downarrow 0} (h^\H,p_u^{\sqrt{\frac{u}{u-\delta}}-1, out}),\end{equation} where all limits are in probability. From this we may conclude, since we already showed that the first limit in \eqref{eq:hpu_def} was a.s.\ equal to $c$, and the right hand side above is equal to the second limit in \eqref{eq:hpu_def} (which we also know exists in probability.)

We will now prove the equalities (a), (b) and (c) from \cref{eqn:letters} in turn. For (a), note that by \cref{lem:harmonic_sines} and scale invariance,  \begin{equation}\label{eqn:a}(\varphi_{\H}^{\mathsf D_{u-\delta}},p_u^{\delta,in})-(\varphi_{\H}^{\mathsf D_{u-\delta}},p_u) \overset{(d)}{=} (\varphi_{\H}^{\D^+}, f_\delta ),\end{equation} where $f_\delta$ are a sequence of uniformly bounded smooth functions supported in vanishing neighbourhoods of $\{ \pm 1\}$. The difference \eqref{eqn:a} therefore converges to $0$ in probability as $\delta\to 0$. Moreover, by Lemma \ref{lem:nested_dmp}, we have

  \[(\varphi_{\H}^{\mathsf D_{u-\delta}},p_u^{\delta,in})-(\varphi_{\H}^{\mathsf D_{u}},p_u^{\delta,in}) \overset{a.s.}{=} (\varphi_{\mathsf D_{u-\delta}}^{\mathsf D_{u}},p_u^{\delta,in}) \overset{(d)}{=} (h^{\mathsf D_{u-\delta}},p_u^{\delta, in} )- (h_{\mathsf D_{u-\delta}}^{\mathsf D_{u}}, p_u^{\delta, in}).\]
Both terms on the right-hand side also converge to $0$ in probability as $\delta\to 0$ by scaling again, and the Dirichlet boundary condition assumption.
Putting these facts together gives (a).

Equality (b) follows by a very similar distributional equality to \eqref{eqn:a}, again using Lemma \ref{lem:harmonic_sines}. Finally (c) holds, since  $$(\varphi_{\H}^{\mathsf D_{u-\delta}},p_u^{\sqrt{\frac{u}{u-\delta}}-1, out})- (h^\H,p_u^{\sqrt{\frac{u}{u-\delta}}-1, out})=-(h_{\H}^{\mathsf D_{u-\delta}}, p_u^{\sqrt{\frac{u}{u-\delta}}-1, out})$$ almost surely  and the right hand side (again by scaling) can be seen to converge to 0 in probability as $\delta \downarrow 0$.
\end{proof}

\section{A characterisation of Brownian motion}\label{sec:char_BM}

\begin{prop}\label{prop:char_BM}
	Suppose that $(Y(u))_{u\in (0,\infty)}$ is a centred stochastic process. For $u>0$, write $\mathcal{F}_u^+:=\sigma(Y_s: s\ge u)$, $\mathcal{F}_u^-:=\sigma(Y_s: s\le u)$, and for $0<s<r$ let $\mathcal{F}_{s,r}$ be the $\sigma$-algebra generated by $\mathcal{F}_s^-$ and $\mathcal{F}_r^+$. Suppose that:
	\begin{enumerate}[(i)]
		\item $(Y(u))_{u\in (0,\infty)}$ is stochastically continuous, i.e., for any $u_0\in (0,\infty)$, $Y_u\to Y_{u_0}$ in probability as $u \to u_0$;
		\item for some $\xi>0$, $\mathbb{E}[|Y(u)|^\xi]<\infty$ for all $u\in (0,\infty)$;
		\item $Y$ satisfies Brownian scaling, that is, $(Y(cu))_{u> 0}$ has the same law as $(\sqrt{c}Y(u))_{u> 0}$ for any $c>0$;
		\item for any $u>0$, $(Y(s)-Y(u))_{s\ge u}$ is independent of $\mathcal{F}_u^-$;
	\item for any $0< s<r$ $(Y(u)-(\frac{u-s}{r-s}Y(r)+\frac{r-u}{r-s}Y(s)))_{u\in (s,r)}$ is independent of $\mathcal{F}_{s,r}$.	\end{enumerate}
	Then there exists a modification of $Y$ that is equal to $\sigma B$ in law for some $\sigma\ge 0$, where $B$ is a standard one-dimensional Brownian motion.
\end{prop}

Observe that for this characterisation we only require $\xi>0$, we will comment later on why we need existence of $(1+\eps)$ moments for the main result of this paper. Also observe that by scaling, for any process $Y$ as in the statement of the proposition, $Y(\delta)$ is equal in distribution to $\sqrt{\delta} Y(1)$ for every $\delta$, and so tends to $0$ in probability as $\delta \to 0$. 

This proposition is very close to the main result of \cite{Wes93}, which is essentially the same but requires square-integrability of the process $Y$. Indeed, we will prove the proposition by showing square-integrability and then appealing to \cite{Wes93}.

We also remark that there is a similar characterisation of Brownian motion in \cite[Theorem 1.9]{BPR18}; the major difference being item $(vi)$. In \cite{BPR18} we assumed that the process in $(vi)$ has the law of a scaled version of the original process. This is stronger than the statement here, which assumes nothing about the law. On the other hand, only finiteness of logarithmic moments was assumed in \cite{BPR18}, which is (slightly) weaker than the moment assumption $(ii)$ above.

For some motivation, let us first see the important corollary of this characterisation for the purposes of the present article. The proof of \cref{prop:char_BM} will follow immediately after.

\begin{corollary}\label{cor:sine_avg_gaussian}	Let $h^\H$ be a sample from $\Gamma^\H$, and define the process $Y$ via \[Y(u):=(h^\H, p_u) \text{ for } u\ge 0,\]
	where the right hand side is as defined in  \cref{prop:def_sine_avg} and
\cref{rmk:hpu_process}. Then $Y$ satisfies the conditions of \cref{prop:char_BM}, and hence has a modification with the law of $\sigma$ times a Brownian motion for some $\sigma\ge 0$.
\end{corollary}

\begin{remark}
 We note that this result actually holds even if we only have $\xi>0$ in Assumption \ref{ass:ci_dmp}, (i). This suggests that the answer to Question \ref{Q:xi} is positive.
\end{remark}

\begin{proof}
Since $Y(u)$ is the $L^1$ limit of $(h^\H,p_u^{\delta,in})$ as $\delta\to 0$, and $(h^\H,p_u^{\delta,in})$ is centred for every $\delta$ and $u$, it follows that $Y$ is a centred process.
So, it suffices to prove the conditions (i)-(vi) of \cref{prop:char_BM}.
		\begin{enumerate}[(i)]
			\item Equality (a) from \eqref{eqn:letters} in the proof of  \cref{prop:def_sine_avg}, plus Lemma \ref{lem:harmonic_sines}, tells us that
			\[ (h^\H,p_1)-(h^\H,p_{1-\delta})\to 0\]
			in probability as $\delta\to 0$. Moreover by scale invariance (see (iii) below) we have that $|(h^\H,p_s)-(h^\H,p_t)|$ is equal in distribution to $\sqrt{s\vee t}\, |(h^\H,p_1)-(h^\H,p_{(s\wedge t)/(s \vee t)})|$. This gives the stochastic continuity.
			\item This holds with $\xi=1$ since $Y(u)$ is defined as a limit in $L^1$ for all $u$.
			\item (Scale invariance) We assume without loss of generality that $c>1$. First, we  claim that
			\begin{equation}\label{eq:SI_harm}
			(z\mapsto \varphi^{\mathsf D_{cu}}_\H(z), z \in \mathsf D_{cu})_{u\ge 0} \text{ and } (z\mapsto \varphi^{\sou}_\H(\sqrt{c}z) , z \in \mathsf D_{cu})_{u\ge 0}
			\end{equation}
			have the same law as processes (of harmonic functions) in $u$, in the sense that the finite dimensional marginals of both sides have the same laws.
			
			The statement for one dimensional marginals is a special case of \cref{lem:harm_ci}. For the higher dimensional marginals, since the argument with $n$ points is very similar, we will just show equality in law for the joint distribution at two points $u < u'$. For this, we use uniqueness of the domain Markov decomposition to write
			$$
			( \varphi_{\H}^{\mathsf D_{cu}} , \varphi_{\H}^{\mathsf D_{cu'}}    ) \overset{(d)}{=} (\varphi_{\H}^{\mathsf D_{cu}} , \varphi_{\H}^{\mathsf D_{cu}} +  \varphi_{\mathsf D_{cu}}^{\mathsf D_{cu'}}    ) \; \text{ and }  \; 	( \varphi_{\H}^{\mathsf D_{u}} , \varphi_{\H}^{\mathsf D_{u'}}    ) \overset{(d)}{=} (\varphi_{\H}^{\mathsf D_{u}} , \varphi_{\H}^{\mathsf D_{u}} +  \varphi_{\mathsf D_{u}}^{\mathsf D_{u'}}    )
			$$
where $\varphi_{\mathsf D_{cu}}^{\mathsf D_{cu'}}$ is independent of $ \varphi_{\H}^{\mathsf D_{cu}}$ and $\varphi_{\mathsf D_{u}}^{\mathsf D_{u'}}$ is independent of $ \varphi_{\H}^{\mathsf D_{u}}$ .  Using this independence, and Lemma \ref{lem:harm_ci} again, we obtain \eqref{eq:SI_harm}.

Now we complete the proof of scale invariance as follows. Fix $u>0$. By definition of the measures $p_u$,
		\begin{eqnarray*} (h^\H, p_{cu}) & \overset{\eqref{eq:hpu_alt_def}}{=} &
	(\varphi_{\H}^{\mathsf D_{cu}},\; p_{2cu})\\ & = & \sqrt{2cu} \int_0^\pi \sin(\theta) \varphi^{\mathsf D_{cu}}_\H (\frac{e^{i\theta}}{\sqrt{2cu}}) \, d\theta \\ &= &\sqrt{c}\sqrt{2u} \int_0^\pi \sin(\theta) \varphi^{\mathsf D_{u}}_\H (\sqrt{c}\frac{e^{i\theta}}{\sqrt{2cu}}) \, d\theta \\ & = & \sqrt{c} (\varphi_{\H}^{\sou},\;p_{2u})\\ & \overset{\eqref{eq:hpu_alt_def}}{=} & \sqrt{c} (h^\H, p_u) 
		\end{eqnarray*}
where we used \cref{eq:SI_harm} in the third equality. Applying the same string of equalities for finite dimensional marginals, 
 we get the result.

			\item Fix $u\ge 0$ and observe that since $Y(s)=\lim_{\delta\downarrow 0}(h^\H,p_s^{\delta,out})=\lim_{\delta \downarrow 0} (\varphi_{\H}^{\sou},p_s^{\delta,out})$ for $s\le u$, $\mathcal{F}_u^-$ is independent of $h_\H^{\sou}$. This means that when we write (see \cref{lem:nested_dmp}) \[\varphi_{\H}^{\sor}=\varphi_{\H}^{\sou}+\varphi_{\sou}^{\sor}\; ; \;\; r\ge u,\] we have that $\varphi_{\sou}^{\sor}$ is independent of $\mathcal{F}_u^-$. Then since
			\[Y(r) \overset{\eqref{eq:hpu_alt_def}}{=}(\varphi_{\H}^{\sor},p_{2r})=(\varphi_{\H}^{\sou},p_{2r})+(\varphi_{\sou}^{\sor},p_{2r})\overset{\eqref{eq:hpu_alt_def}}{=}Y(u)+(\varphi_{\sou}^{\sor},p_{2r}),\]
			we reach the desired conclusion.
		\item Let us write $A_{r,s}:= \sos\setminus\overline{\sor}$. Reasoning as in the proof of (iv), we see that in the decomposition $$h^\H=h_\H^{A_{r,s}}+\varphi_\H^{A_{r,s}},$$ $h_\H^{A_{r,s}}$ is independent of $\mathcal{F}_{s,r}$. Hence, we must argue that 
		\begin{equation} \label{eq:harness_cond}(\varphi_\H^{A_{r,s}},p_u)=\frac{u-s}{r-s}Y(r)+\frac{r-u}{r-s}Y(s) \text{ for all } u\in (s,r).\end{equation}
		Now, by \cref{lem:harmonic_sines} we know that the left hand side of \eqref{eq:harness_cond} is a.s.\,a linear function of $u\in (s,r)$, so we just need to prove that its limit as $u\downarrow s$ is equal to $Y(s)$, and as $u\uparrow r$ is equal to $Y(r)$.
		
		Let us prove the first limit, the second one being very similar.
		For this, write
		\[ \lim_{u\downarrow s} (\varphi_\H^{A_{r,s}},p_u)=\lim_{u\downarrow s} (\varphi_\H^{\sos},p_u)+\lim_{u\downarrow s}(\varphi_{\sos}^{A_{r,s}}, p_u)=Y(s)+\lim_{u\downarrow s}(\varphi_{\sos}^{A_{r,s}}, p_u)\]
		and observe that by \cref{ass:ci_dmp} (iv), \[\varphi_{\sos}^{A_{r,s}} \text{ is harmonic in } A_{r,s} \text{ and goes to zero on } \partial(\sos)\cup(-\frac{1}{\sqrt{s}},-\frac{1}{\sqrt{r}})\cup(\frac{1}{\sqrt{r}},\frac{1}{\sqrt{s}}).\]
		This implies that $|\varphi_{\sos}^{A_{r,s}}|$ is uniformly bounded in a neighbourhood of $\partial(\sos)$ in $\sos$, and hence, by dominated convergence, we deduce that $\lim_{u\downarrow s}(\varphi_{\sos}^{A_{r,s}}, p_u)=0.$
			\end{enumerate}
\end{proof}

\begin{proof}[Proof of \cref{prop:char_BM}]
	This almost follows from  \cite[Theorem 1]{Wes93}, except for the square integrability condition. So first, we will prove that
		\begin{equation}
		\label{eqn:second_moment}
	\mathbb{E}[|Y(u)|^2]<\infty \;\; \forall u\in [0,\infty).\end{equation}
		To do this, pick some $n$ such that $2^{-n}\le \xi$, so that by assumption $\mathbb{E}[|Y(u)|^{2^{-n}}]<\infty$ for all $u$. We will prove that for any $m\ge 0$,
	\begin{equation}\label{eq:induction}
	\mathbb{E}[|Y(u)|^{2^{-m}}]<\infty \;\; \forall u\in [0,\infty) \; \Rightarrow \mathbb{E}[|Y(u)|^{2^{-m+1}}]<\infty \;\; \forall u\in [0,\infty),
	\end{equation} from which the result follows by induction, starting with $m=n$.

So, let us take some $m\ge 0$ and assume that the left hand side of \eqref{eq:induction} holds. Denote $\eta:=2^{-m}$ and first observe that $\mathbb{E}[|Y(2)-Y(1)|^\eta]<\infty$, since $|x+y|^\eta\le |x|^\eta + |y|^\eta$. By independence of $(Y(2)-Y(1))$ and $Y(1)$ (condition (iv) of \cref{prop:char_BM}), this implies that $\mathbb{E}[|Y(1)(Y(2)-Y(1))|^\eta]<\infty$. Now we apply condition (v) of \cref{prop:char_BM}. Applying this with $s=\delta, u=1, r=2$ for any $\delta\in (0,1)$ tells us that we can write $Y(1)=\frac{1-\delta}{2-\delta}Y(2)+\frac{1}{2-\delta}Y(\delta)+Z(\delta)$, where $(Z(\delta))_{ \delta\in(0,1)}$ is independent of $Y(2)$. Sending $\delta$ to $0$ (and using, as noted before, that $Y(\delta)\to 0$ in probability as $\delta\to 0$) implies that $Y(1)=Y(2)/2+Z$, where $Z$ is independent of $Y(2)$. Hence \[Y(1)(Y(2)-Y(1))=(\frac{Y(2)}{2}+Z)(\frac{Y(2)}{2}-Z)=\frac{Y(2)^2}{4} - Z^2\] has a finite moment of order $\eta$. Applying \cref{lem:indep_moments}, we obtain that $|Y(2)|^2$ has a finite moment of order $\eta$, and hence by scale invariance (condition (iii) of \cref{prop:char_BM}), that $\mathbb{E}[|Y(u)|^{2\eta}]<\infty$ for all $u\in [0,\infty)$. This completes the proof of the induction step, \eqref{eq:induction}, and therefore of \eqref{eqn:second_moment}.

From here, we can appeal to the characterisation in \cite[Theorem 1]{Wes93} of stochastic processes with linear conditional expectation and quadratic conditional variance. This says that if $Y$ is a process as in \cref{prop:char_BM}, that in addition
\begin{itemize}
	\item {is defined and stochastically continuous on $[0,\infty)$ with $Y(0)=0$,}
	\item has $Y(u)$ square integrable for every $u$,
	\item has $\mathbb{E}[Y(u)Y(s)]=\mathbb{E}[Y(u\wedge s)^2]=\sigma (u\wedge s)$ for some $\sigma\ge 0$ and all $u,s\in [0,\infty)$
\end{itemize}
then $Y$ must be $\sigma$ times a standard Brownian motion.
{Note that by the discussion immediately after the statement of \cref{prop:char_BM}, we can extend $Y$ to a stochastically continuous process on $[0,\infty)$ with $Y(0)=0$.} We also get the third point above by the assumption of Brownian scaling, plus the fact that the process is centred with independent increments. Hence \cite[Theorem 1]{Wes93} provides the result.
\end{proof}

\section{Gaussianity of circle averages}\label{sec:circ_avg_gaussian}
In this section we work with a sample $h^\D$ from $\Gamma^\D$. For any $\eps>0$ we can define the circle average $h_\eps(0)$ at radius $\eps$ around $0$ via
\[  h^\D_\eps(0):= \varphi_\D^{B_0(\eps)}(0)\]
as in \cite{BPR18}.
Our next goal is to relate these circle averages to the sine averages from Section \ref{sec:sine_avgs}. This will allow us to show (using \cref{cor:sine_avg_gaussian}) that the circle average process possesses a modification that is continuous in $\eps$, and will in turn imply that $(h^\D_{\e^{-t}}(0))_{t\ge 0}$ (which has independent and stationary increments by conformal invariance and the domain Markov property) is a Brownian motion. From this it will follow that $h^\D_\eps(0)$ is Gaussian for any $\eps>0$.

To begin, we will explain how the sine averages from \cref{sec:sine_avgs} can make sense for $h^D$ with some specific domains $D\ne \H$. Essentially, this is due to the domain Markov property, which allows us to relate $h^D$ with $h^\H$ in such a way that the sine average of one is the sine average of the other plus the sine average of a harmonic function.

For example, let us start with $D=\D^+$. By the domain Markov property, we can decompose $h^\H$ in the upper unit semi disc $\D^+$ as the independent sum $$h^\H=h_\H^{\D^+}+\varphi_\H^{\D^+},$$ and we already know that:
\begin{itemize}
		\item for any $u\ge 1$, $(h^{\H},p_u^{\delta, in})\to (h^\H,p_u)$ in probability and in $L^1$ as $\delta\to 0$;
		\item for any $u>1$, $(\varphi_{\H}^{\D^+},p_u^{\delta,in})\to (\varphi_{\H}^{\D^+},p_u) \text{ a.s.\ and in } L^1 \text{ as } \delta\to 0$, where  $(\varphi_\H^{\D^+},p_u)$ is a.s.\ constant in $u>1$;
		\item $(\varphi_{\H}^{\D^+},p_1^{\delta,in})$ converges to this constant value in probability and in $L^1$ as $\delta\to 0$ (using \eqref{eq:vp_const} and the argument explained just after).

\end{itemize}

For the first bullet point we have used \cref{prop:def_sine_avg}, and for the second, \cref{lem:harmonic_sines} plus the fact that $\varphi_{\H}^{\D^+}$ is almost surely harmonic in $\D^+$ and can be extended continuously to $0$ on $(-1,1)$ (\cref{lem:harm_0_boundary}).

This implies that for each $u\ge 1$, $$\lim_{\delta\to 0}(h_\H^{\D^+},p_{u}^{\delta, in})=:(h^{\D^+},p_u)$$ exists in probability and in $L^1$. Similarly, the joint limit 
$\lim_{\delta\to 0}((h_\H^{\D^+},p_{u_1}^{\delta, in}),\dots, (h_\H^{\D^+},p_{u_n}^{\delta, in}))$ 
exists in probability and in $L^1$ for any $(u_1,\cdots, u_n)$ with each $u_i\in [1,\infty)$. Notice that, by the above observations, the limit of such a vector must be equal in law to $((h^\H,p_{u_1}),\dots,(h^\H,p_{u_n}) )$ plus the (random) vector $((\varphi_{\H}^{\D^+},p_{u_1}),\dots, (\varphi_{\H}^{\D^+},p_{u_n}))$, whose components are almost surely all equal. Notice further that
$(h_{\H}^{\D^+},p_1^{\delta,in})\to 0$ in $L^1$ and in probability as $\delta\downarrow 0$ (by the Dirichlet boundary condition assumption), so that $(h^{\D^+},p_1)=0$.

Putting all this together with \cref{cor:sine_avg_gaussian}, we obtain the following:

\begin{lemma}\label{lem:bmdplus} Let $h^{\D^+}$ be a sample from $\Gamma^{\D^+}$. Then for any $(u_1,\cdots, u_n)$ with $u_i\in [1,\infty)$ for $1\le i \le n$, the limit
	\[ \lim_{\delta\downarrow 0}\left((h^{\D^+},p_{u_1}^{\delta,in}),\ldots, (h^{\D^+},p_{u_n}^{\delta,in})\right)=\left((h^{\D^+},p_{u_1}),\ldots,(h^{\D^+},p_{u_n}) \right) \]
	exists in probability. Moreover, $(h^{\D^+},p_{1+t})_{t\ge 0}$ has the same finite dimensional distributions as some multiple (which is the same as that in \cref{cor:sine_avg_gaussian}) of Brownian motion.
\end{lemma}

Next, we make sense of sine averages for $h^{\D}$. Again we can use the domain Markov property, and decompose \begin{equation}
\label{eqn:hDdecomp}
h^\D=h_\D^{\D^+}+\varphi_{\D}^{\D^+}. \end{equation}
However, deducing something from this is not quite so simple, since $\varphi_{\D}^{\D^+}$ does \emph{not} extend continuously to $0$ on $(-1,1)$. For example, since $(\varphi_{\D}^{\D^+},p_u)$ should correspond to integrating $\varphi_{\D}^{\D^+}$ on a contour that \emph{does} touch the real line, it is not immediately obvious that this integral is well defined. We can manage this using that \begin{itemize}  \item[(a)] $\varphi_{\D}^{\D^+}$ is not too badly behaved, and \item[(b)] the density $\sin(\theta)$ converges to $0$ as $\theta\to \{0,\pi\}$.\end{itemize} For this some quantitative estimates are required, and we summarise them in the following lemma:

\begin{lemma}
	\label{lem:phi_der} There exists a universal constant $C\in (0,\infty)$, such that for all $\eps>0$,
	\begin{equation}\label{eqn:bound_sup_phi} \mathbb{E}[\sup_{w\in \D^+;\, \Im(w)>\eps}|\varphi_{\D}^{\D^+}(w)|]\le C \eps^{-1/\xi} \log(1/\eps)^{1/\xi}; \text{ and }\end{equation}

	\begin{equation}\label{eqn:bound_der_phi} \mathbb{E}[\sup_{r\in [0,1],\theta\in [0,\pi]; \, \Im(r\e^{i\theta})>\eps}|\frac{\partial}{\partial r}\varphi_{\D}^{\D^+}(r\e^{i\theta})|]\le C \eps^{-1-1/\xi}\log(1/\eps)^{1/\xi},\end{equation}
	where $\xi>1$ is such that $\mathbb{E}[|(h^D,\phi)|^\xi]<\infty$ for all $D$ and $\phi\in C_c^\infty(D)$ (\cref{ass:ci_dmp}(i)).
\end{lemma}

\begin{proof}
	It is a standard fact (a consequence of, e.g., \cite[\S 2.2, Theorem 7]{Eva98}) that for a universal $C'>0$, for any function $\varphi$ that is harmonic in $B_z(r)\subset \C$ and for any $\mathbf{v}$ with modulus $1$,  $|\partial_{\mathbf{v}}\varphi(z)|\le (C'/r) \sup_{y\in B_z(r)}|\varphi(y)|$. Hence \eqref{eqn:bound_der_phi} follows from \eqref{eqn:bound_sup_phi}.
	
	To prove \eqref{eqn:bound_sup_phi}, let $w\in \D^+$ with $\Im(w)>\eps$ be arbitrary, and denote by $D_\eps$ the domain $\D^+\cap\{z: \Im(z)>\eps/2\}$. Let $a_\eps=\sqrt{1-\eps^2/4}$, and for $y\in [-a_\eps,a_\eps]$, let $f_w(y)$ be the density at $y+i\eps/2$ of the exit position from $D_\eps$ for a Brownian motion started from $w$. Then by harmonicity and the fact that $\varphi_{\D}^{\D^+}$ extends continuously to $0$ on $\partial D_\eps \cap \partial \D$ (by \cref{lem:harm_0_boundary}, conformal invariance and the domain Markov property) we have that
	\[ \varphi_{\D}^{\D^+}(w)= \int_{-a_\eps}^{a_\eps} f_w(y)\varphi_{\D}^{\D^+}(y+i \eps/2) \, dy. \]
This implies, using H\"{o}lder's inequality, that 

		\[ |\varphi_{\D}^{\D^+}(w)|\le \left(\int_{-a_\eps}^{a_\eps} f_w(y)dy\right)^{1/\xi^*} \left(\int_{-a_\eps}^{a_\eps} f_w(y)|\varphi_{\D}^{\D^+}(y+i \eps/2)|^\xi \, dy\right)^{1/\xi} \]
		where $\xi^*$ is such that $1/\xi+1/\xi^*=1$.
		Moreover, by domination with respect to a Cauchy density, there exists a constant $M$ not depending on $\eps>0$, such that $$0\le f_w(y)\le M/\eps \quad \forall y\in [-1,1]\, , w\in D_{2\eps}.$$ Putting this together, along with the fact that $\int_{-a_\eps}^{a_\eps} f_w(y) \, dy \le 1$, we obtain that
			\[ \sup_{w\in \D^+;\, \Im(w)>\eps} |\varphi_{\D}^{\D^+}(w)|^\xi \le  \frac{M}{\eps} \int_{-a_\eps}^{a_\eps} |\varphi_{\D}^{\D^+}(y+i \eps/2)|^\xi \, dy. \]
To conclude, we observe that by \cref{lem:moment_bound}
\[ \mathbb{E}[|\varphi_{\D}^{\D^+}(y+i \eps/2)|^\xi]\le C'' \log(1/\eps) \;\; \forall y\in [-a_\eps,a_\eps],\]
with  constant $C''$ not depending on $\eps>0$, so that
\[ \mathbb{E}[\sup_{w\in \D^+;\, \Im(w)>\eps} |\varphi_{\D}^{\D^+}(w)|]\le\mathbb{E}[\sup_{w\in \D^+;\, \Im(w)>\eps} |\varphi_{\D}^{\D^+}(w)|^\xi]^{1/\xi}\le C \eps^{-1/\xi} \log(1/\eps)^{1/\xi} \]
for some universal constant $C$, as required.
\end{proof}

This allows us to deduce the following:
\begin{lemma}\label{lem:cont_sin_avg}
Let $h^\D$ be a sample from $\Gamma^\D$ and recall the decomposition \eqref{eqn:hDdecomp}. Then for each $(u_1,\cdots, u_n)$ with  $u_i\in [1,\infty)$ for $1\le i \le n$ the limit
\begin{equation}\label{eqn:sa_hd_1} \lim_{\delta\downarrow 0}\left((h^{\D^+}_{\D},p_{u_1}^{\delta,in}),\ldots,(h^{\D^+}_{\D},p_{u_n}^{\delta,in})\right) =:\left((h^{\D^+}_\D,p_{u_1}),\ldots,(h^{\D^+}_\D,p_{u_n})\right)  \end{equation} exists in probability, and the resulting finite dimensional distributions are those of a multiple (which is the same as that in \cref{cor:sine_avg_gaussian}) of Brownian motion. Furthermore, on an event of probability one,
\begin{equation}\label{eqn:sa_hd_2} \left((\varphi^{\D^+}_{\D},p_{u}^{\delta,in})\right)_{u\ge 1} \text{  has a pointwise (in u) limit } \left((\varphi^{\D^+}_\D,p_{u})\right)_{u\ge 1} \text{ as } \delta\to 0, \end{equation}
and this limit is a continuous function. Finally, for any $1\le v<w<\infty$, there exists $M(v,w)$ such that,
	\begin{equation}\label{eqn:sa_hd_3}\mathbb{E}[ \sup_{s,t\in [v,w]}\frac{|(\varphi_{\D}^{\D^+},p_s)-(\varphi_{\D}^{\D^+},p_t)|}{|s-t|}] \le M(v,w) .\end{equation}
\end{lemma}

\begin{remark}
		In words, this tells us that the sine-average process of $h^\D$ (defined by joint limits of $(h^\D,p_u^{\delta, in})$ as $\delta\to 0$) makes sense and is a Brownian motion plus a nicely behaved continuous function whose derivative is bounded in expectation, \eqref{eqn:sa_hd_3}. The role of this key lemma is to show that when we ``average" the sine-average process over rotations (as will soon be made precise) we obtain a process with a continuous modification. The control given by \eqref{eqn:sa_hd_3} is important here to ensure that we retain continuity after averaging, and it is for this that we need the existence of moments with order strictly greater than $1$ (we remark that we have also used it in several other places for simplicity).
		
		 This is really the crux of the proof, since the resulting ``averaged'' process will actually turn out to be the circle average process for $h^\D$ around $0$ (recall from the introduction that establishing continuity of circle averages is the main step in our argument).
\end{remark}

\begin{proof}
Since $h_{\D}^{\D^+}$ has the same law as $h^{\D^+}$, the statement concerning the limit \eqref{eqn:sa_hd_1} follows from \cref{lem:bmdplus}.
To show that \eqref{eqn:sa_hd_2} holds with probability one note that by Markov's inequality, for any $\xi^{-1}<a<1$,
$$
\mathbb{P}[\sup_{w\in \D^+;\, \Im(w)>\eps}|\varphi_{\D}^{\D^+}(w)| > \ve^{-a}]\le C \eps^{a-1/\xi} \log(1/\eps)^{1/\xi}
$$
Thus applying the Borel--Cantelli lemma (to the sequence $\ve_n = 2^{-n}$) we conclude that almost surely, for any $\xi^{-1}<a<1$, $$|\varphi_{\D}^{\D^+}(z)|\le \Im(z)^{-a}$$ for all $z\in \D^+$ with $\Im(z)$ sufficiently small. This implies \eqref{eqn:sa_hd_2} (since $\sin(\arg(z))\Im(z)^{-a}\to 0$ as $\Im(z)\to 0$).
Similarly, an application of the Borel--Cantelli lemma and \eqref{eqn:bound_der_phi} allows us to deduce that for any $1+\xi^{-1}<b<2$, on an event of probability one,
$$|\frac{\partial}{\partial r}\varphi_{\D}^{\D^+}(r\e^{i\theta})|\le \Im(z)^{-b}$$ 
 for all $z\in \D^+$ with $\Im(z)$ sufficiently small. On this event, since $\int_{0}^\pi \sin(\theta)^{1-b}<\infty$, $F(u):=(\varphi_{\D}^{\D^+},p_u)$ is differentiable in $u$, and for some finite deterministic constants $\{M'(v,w)\}_{1<v<w<\infty}$, \[ |F'(r)|\le M'(v,w) \int_0^\pi \sin(\theta) |\frac{\partial}{\partial r}\varphi_{\D}^{\D^+}(\e^{i\theta}/\sqrt{r})| \, d\theta \text{ for all } r\in [v,w] \]
From this and \eqref{eqn:bound_der_phi}, \eqref{eqn:sa_hd_3} follows in a straightforward manner.
\end{proof}

Now we will relate these quantities to circle averages, by averaging over rotations. Let $h^\D$ be a sample from $\Gamma^\D$ and for $\alpha \in [0,2\pi)$, let $h^{\D,\alpha}$ be the image of $h^\D$ under an anti-clockwise rotation by angle $\alpha$. That is, $(h^{\D,\alpha},\phi)_{\phi\in C_c^\infty(\D)}=(h^{\D},\phi\circ f_\alpha)_{\phi\in C_c^\infty(\D)}$ where $f_\alpha$ denotes the isometry $z\mapsto \e^{-i\alpha}z$.

 Then by conformal (specifically, rotation) invariance, \begin{equation}
\label{eq:law_ind_a} h^{\D,\alpha}\overset{(d)}{=}h^\D\end{equation} for each fixed $\alpha$. Write $h^{\D^+}_{\D,\alpha}+\varphi^{\D^+}_{\D,\alpha}$ for the domain Markov decomposition of $h^{\D,\alpha}$ in $\D^+$.

Now let $A$ be uniformly distributed on the interval $[0,2\pi]$ (independently from $h^\D$). Then we have that:
\begin{itemize}
	\item for each $(u_1,\cdots, u_n)$ with  $u_i\in [1,\infty)$ for $1\le i \le n$ $$\lim_{\delta\downarrow 0}\left((h^{\D^+}_{\D,A},p_{u_1}^{\delta,in}),\ldots, (h^{\D^+}_{\D,A},p_{u_n}^{\delta,in})\right)=:\left((h^{\D^+}_{\D,A},p_{u_1}),\cdots, (h^{\D^+}_{\D,A},p_{u_n})\right)$$ exists a.s.\ and for any $s,t\ge 1$ \begin{equation}\label{eqn:fourth_moment_circ} \mathbb{E}[|(h_{\D,A}^{\D^+},p_s)-(h_{\D,A}^{\D^+},p_t)|^4]\le c|s-t|^2\end{equation} for some universal constant $c$ (because for each angle $\alpha$ the process $(h^{\D^+}_{\D,\alpha},p_s)_s$ is a fixed, i.e. not depending on $\alpha$, multiple of Brownian motion);
	\item  $((\varphi^{\D^+}_{\D,A},p_{u}^{\delta,in}))_{u\ge 1}$ has a pointwise limit $((\varphi^{\D^+}_{\D,A},p_{u}))_{u\ge 1}$ with probability one as $\delta\to 0$, and for any $1<v<w<\infty$, there exists $M(v,w)$ such that,
	\begin{equation}\label{eqn:circ_der_bound}\mathbb{E}[ \sup_{s,t\in [v,w]}\frac{|(\varphi_{\D,A}^{\D^+},p_s)-(\varphi_{\D,A}^{\D^+},p_t)|}{|s-t|}] \le M(v,w) .\end{equation}
\end{itemize}

This allows us to reach the following conclusion.

\begin{lemma}
	For every $u\in [1,\infty)$, the conditional expectation $$\mathbb{E}[(h^{\D,A},p_u) \, | \, h^\D ]:=\mathbb{E}[(h_{\D,A}^{\D^+},p_u)+(\varphi_{\D,A}^{\D^+},p_u) \, | \, h^\D ]$$ is well defined. This defines a stochastic process in $u$ which possesses an a.s.\ continuous modification.
\end{lemma}

\begin{proof}
	Since $(h_{\D,A}^{\D^+},p_u)$ and $(\varphi_{\D,A}^{\D^+},p_u)$
are random variables in $L^1(\P\times dA)$ (as can be seen using \eqref{eq:law_ind_a}, by first taking expectation over the field given $A$, and then over $A$) the conditional expectations $$\mathbb{E}[(h_{\D,A}^{\D^+},p_u) \, | \, h^\D ] \text{ and } \mathbb{E}[(\varphi_{\D,A}^{\D^+},p_u) \, | \, h^\D ]$$ are  well defined for any fixed $u$. By \eqref{eqn:fourth_moment_circ}, the fact that conditioning is a contraction in $L^4$, and Kolmogorov's continuity criterion, the first of these two stochastic processes has an a.s.\ continuous modification. To deal with the second process, observe that by \eqref{eqn:circ_der_bound} and Jensen's inequality, for any $1<v<w<\infty$, we have
\begin{eqnarray*}  & \mathbb{E}\left[\sup_{s,t\in [v,w]} \frac{\left|\mathbb{E}[(\varphi_{\D,A}^{\D^+},p_t) \, | \, h^\D]-\mathbb{E}[ (\varphi_{\D,A}^{\D^+},p_s)\, | \, h^\D]\right|}{|s-t|}\right] \\ & \le  \mathbb{E}\left[ \mathbb{E}[\sup_{s,t\in [v,w]}\frac{|(\varphi_{\D,A}^{\D^+},p_t) - (\varphi_{\D,A}^{\D^+},p_s)|}{|s-t|}\, | \, h^\D]\right] \le  M(v,w).
\end{eqnarray*}
Hence the process $\mathbb{E}[(\varphi_{\D,A}^{\D^+},p_u) \, | \, h^\D ]$ in $u$ has a modification which is a.s. continuous. \end{proof}

The connection to circle averages is the following. Recall that $h_\eps^\D(0)$ denotes the radius $\eps$ circle average of $h^\D$ around $0$. Recall that this is defined to be equal to $\varphi_{\D}^{\eps\D}(0)$ if $h^\D$ has domain Markov decomposition $h^{\eps\D}_\D+\varphi_{\D}^{\eps\D}$ in $\eps\D$.
\begin{lemma}
	\label{lem:cond_sin_av_equals_circ_av}
	For any $u\in [1,\infty)$,
	$\mathbb{E}[(h^{\D,A},p_u) \, | \, h^\D]=\sqrt{u}h^\D_{\frac{1}{\sqrt{u}}}(0)$ a.s.
\end{lemma}

\begin{proof}
Fix $u\in [1,\infty)$. Since $(h^{\D,A},p_u^{\delta,in})\to (h^{\D,A},p_u)$ in probability and in $L^1$ as $\delta\to 0$, we have that \[\mathbb{E}[(h^{\D,A},p_u) \, | \, h^\D]=\mathbb{E}[\lim_{\delta\downarrow 0}(h^{\D,A},p_u^{\delta,in}) \, | \, h^\D]=\lim_{\delta\downarrow 0} \mathbb{E}[(h^{\D,A},p_u^{\delta,in})\, | \, h^\D]\] where the rightmost limit holds in probability and in $L^1$. By definition of $A$, the right hand side is equal to \[ \lim_{\delta\downarrow 0}\frac{1}{2\pi}\int_0^{2\pi} (h^{\D,\alpha},p_u^{\delta,in}) \, d\alpha = \lim_{\delta\downarrow 0}\frac{1}{2\pi}\int_0^{2\pi} (h^{\D},p_u^{\delta,in}\circ f_\alpha) \, d\alpha \]
where $f_\alpha(z)=e^{-i\alpha}z$ is rotation by $\alpha$. By linearity of $h^\D$ this is equal to
\[\lim_{\delta\downarrow 0}(h^\D, \frac{1}{2\pi} \int_0^{2\pi} p_u^{\delta,in}\circ f_\alpha \, d\alpha)=\lim_{\delta\downarrow 0}(\varphi_{\D}^{\frac{1}{\sqrt{u}}\D},\frac{1}{2\pi}\int_0^{2\pi} p_u^{\delta,in}\circ f_\alpha \, d\alpha)+\lim_{\delta\downarrow 0}(h_{\D}^{\frac{1}{\sqrt{u}}\D},\frac{1}{2\pi}\int_0^{2\pi} p_u^{\delta,in}\circ f_\alpha \, d\alpha),\]
where the second term above goes to $0$ in probability as $\delta\to 0$ by the Dirichlet boundary condition assumption. Moreover, the function $\frac{1}{2\pi}\int_0^{2\pi} p_u^{\delta,in}\circ f_\alpha \, d\alpha$ is radially symmetric with total mass tending to $\sqrt{u}$ as $\delta\to 0$. By harmonicity, it then follows that
\[\lim_{\delta\downarrow 0}(\varphi_{\D}^{\frac{1}{\sqrt{u}}\D},\frac{1}{2\pi}\int_0^{2\pi} p_u^{\delta,in}\circ f_\alpha \, d\alpha)= \sqrt{u} \varphi_{\D}^{\frac{1}{\sqrt{u}}\D}(0)=\sqrt{u}h^\D_{\frac{1}{\sqrt{u}}}(0)\]
a.s., as required.
\end{proof}
We emphasise that the process in \cref{lem:cond_sin_av_equals_circ_av} above is not Brownian motion, but rather a time change of it. The corollary is the following:

\begin{corollary}\label{cor:circ_avg_cont}
The process	$(h^\D_\eps(0))_{\eps\in (0,1]}$ possesses a continuous modification.
\end{corollary}

\begin{prop}\label{prop:circ_av_bm}
	The process $(h_{e^{-t}}^\D(0))_{t\ge 0}$ has a modification whose law is that of $(\sigma B_t)_{t\ge 0}$, where $\sigma\ge 0$ and $B$ is a standard one-dimensional Brownian motion.
\end{prop}

\begin{proof}
By the assumptions of conformal invariance and the domain Markov property, this process has independent increments, and it is also centred. By \cref{cor:circ_avg_cont}, it possesses a continuous modification. Since any continuous centred L\'{e}vy process must be a multiple of Brownian motion, this implies the result.
\end{proof}

\begin{corollary}\label{cor:circ_avg_gaussian}
	For any $D$ and $z\in D$, let $F_z^D$ be the conformal map from $D\to \D$ with $z\mapsto 0$ and $(F_z^D)'(z)\in \R_+$.
	Then the process
	\begin{equation} \label{eqn:he} \hat{h}_{e^{-t}}^D(z):=\varphi_D^{(F_z^D)^{-1}(B_0(e^{-t}))}(z)
	\end{equation}
	defined for $t\ge 0$, has a modification whose law is that of $\sigma$ times a Brownian motion.
\end{corollary}

\begin{proof}
This follows from conformal invariance, \cref{ass:ci_dmp}(iii).
\end{proof}
\section{Conclusion of the proof}

\begin{proof}[Proof of Proposition \ref{prop:fourth_moment}\,(1)]
	Without loss of generality we assume that $D=\D$. For $z\in \D$ and $\eps = \eps(z) <d(z,\partial \D)=d(z,\partial D)$. Let \begin{equation}
		\label{eq:rz}
		r_z(\eps):=\sup\{r\in [0,1]\, :  (F_z^\D)^{-1}(B_0(r)) \subset B_z(\eps)\}.\end{equation}
	Also set $h^\D_{\eps}(z)=\varphi_{\D}^{B_z(\eps)}(z)$ and define $\hat{h}^\D_{r_z(\eps)}(z)$ via \eqref{eqn:he} and \eqref{eq:rz}.

	For $\delta>0$, define $\eta_\delta$ to be a smooth radially symmetric function that approximates uniform measure on the unit circle as $\delta\to 0$. For concreteness, $\eta_\delta$ can be taken to be a smooth radially symmetric function equal to 1 on the annulus $\{z: 1 - \delta \le |z| \le 1 -\delta/2 \}$ that is 0 outside a $\delta/10$ neighbourhood of this annulus. We assume that each $\eta_{\delta}$ is normalised to have total integral one. For $\eps\in (0,1)$, further define $$\eta^\eps_\delta(\cdot):=\frac{1}{\eps^2} \eta_{\delta}(\frac{\cdot}{\eps})$$
	
	Take $\phi\in C_c^\infty(\D)$. Recall that for Proposition \ref{prop:fourth_moment}(1) we need to show that $(h^\D, \phi)$ has finite fourth moment. The idea is to show that \begin{equation}\label{eq:fourth_moment_1}
	\int_{\D} \hat{h}_{r_\eps(z)}^\D(z) \phi(z) \, dz \to (h^\D, \phi) \text{ in probability as } \eps\to 0 \end{equation} and that \begin{equation}\label{eq:fourth_moment_two} \left(\int_{\D} \phi(z) \hat{h}^\D_{r_\eps(z)}(z)\, dz
\right)^4 \text{ is uniformly integrable in } \eps\end{equation} This means that $(\int_{\D} \phi(z) \hat{h}^\D_{r_\eps(z)}(z))^4$ converges in $L^1$ to $(\phi,h^\D)^4$, and in particular, that $(\phi,h^\D)^4$ is integrable.
	\vspace{0.2cm}
	
	\textit{Proof of \eqref{eq:fourth_moment_1}.} We bound, for $\delta>0$:
	
	\begin{eqnarray}\label{eq:triangle}
	&\left|	\int \hat{h}_{r_\eps(z)}^\D(z) \phi(z)\, dz - (h^\D,\phi)\right| \nonumber \\
 \le &   \left| \int (\hat{h}_{r_\eps(z)}^\D(z)-h^\D_{\eps}(z))\phi(z) \, dz \right| + \left| \int h^\D_{\eps}(z)\phi(z)\, dz - (h^\D, \phi*\eta^\eps_\delta)\right| +
	\left| (h^\D, \phi*\eta^\eps_\delta)-(h^\D,\phi)\right|
	\end{eqnarray}

	We start by showing that the first term in \eqref{eq:triangle} goes to 0 in probability as $\eps\to 0$. For this, observe that the conformal map $F_z^\D$ can be defined by $F_z^\D(w)=(z-w)/(1-\bar{z}w)$. Hence for $\delta<d(z,\partial \D)$ we have that $$|F_z^\D(w)|\le \frac{\delta}{1-|z|^2+\delta} \Rightarrow |w-z|\le \frac{\delta(1-|z|^2+|z||z-w|)}{1-|z|^2+\delta}\le \delta \text{ and so } r_z(\delta)\ge \frac{\delta}{1-|z|^2+\delta}.$$ On the other hand, $$|y|=\frac{|(F_z^\D)^{-1}(y)-z|}{|1-\bar{z}(F_z^\D)^{-1}(y)|}\ge \frac{\delta}{1-|z|^2+\delta} \Rightarrow |(F_z^\D)^{-1}(y)-z|\ge \delta\, \frac{1-|z|^2-\delta}{1-|z|^2+\delta},$$
	which therefore implies that $(F_z^\D)^{-1}(B_0(r_z(\delta)))$ contains the ball of radius  $\delta(1-2\delta(1-|z|^2+\delta)^{-1})$ around $z$. 
	
Thus, by conformal invariance and \cref{lem:nested_dmp}, $$h_\delta^\D(z)-\hat h_{r_z(\delta)}^\D(z)\overset{(d)}{=}\varphi_{\D}^{D_\delta^z}(0),$$ where for some $f(\delta)$ tending to $0$ as $\delta\to 0$ and every $z$ in the support of $\phi$, $D_\delta^z\subset \D$ contains the ball of radius $1-f(\delta)$ around 0.
	By \eqref{eq:mono_moments}, it then follows that $$\mathbb{E}[|h_\delta^\D(z)-\tilde{h}_{r_z(\delta)}^\D(z)|]\le \mathbb{E}[|\varphi_{\D}^{B_0(1-f(\delta))}(0)|]=  \mathbb{E}[|h_{(1-f(\delta))}^\D(0)|],$$ and this tends to $0$ as $\delta\to 0$ by \cref{prop:circ_av_bm}. By boundedness of $\phi$, this proves that the first term of  \eqref{eq:triangle} goes to 0 in probability as $\eps\to 0$.
	
	We also have that the third term of \eqref{eq:triangle} goes to 0 in probability as $\eps\to 0$, for any fixed $\delta$. Indeed, $\phi*\eta_{\delta}^\eps \to \phi$ in $C_c^\infty(\D)$ as $\eps\to 0$ because $\eta_{\delta}$ is a smooth approximation to the identity for every $\delta$: see, eg. \cite[\S 5.3]{Eva98}. Thus by \cref{ass:ci_dmp}(i) (stochastic continuity), $(h^\D,\phi*\eta^\eps_\delta)\to (h^\D,\phi)$ in probability as $\eps\to 0$.
	
	So to show \eqref{eq:fourth_moment_1} we are left to prove that the middle term of \eqref{eq:triangle} goes to $0$ in probability as $\delta\to 0$, \emph{uniformly} in $\eps$. That is, for any $c>0$ the probability that this term is bigger than $c$ goes to $0$ as $\delta\to 0$, uniformly in $\eps$. To do this, we note that $\phi*\eta_\delta^\eps(z)=\int \phi(w) \eta_{\delta}^\eps(w-z) \, dw$ and so by linearity of $h^\D$,
	$$ (h^\D,\phi*\eta^\eps_\delta)=\int_w (h^\D, \eta_{\delta}^\eps(w-\cdot))\phi(w) \, dw. $$
	Moreover, by the Dirichlet boundary condition assumption and scale invariance, for every $w$ in the support of $\phi$ $$(h^\D,\eta_{\delta}^\eps(w-\cdot))-h^\D_{\delta}(w)\to 0$$ in probability and in $L^1$ as $\delta\to 0$, uniformly in $\eps$. Combined with the boundedness of $\phi$, this completes the proof.

	\vspace{0.2cm}
	\textit{Proof of \eqref{eq:fourth_moment_two}.} For this, we will show that $\int_{\D} \phi(z) \hat{h}^\D_{r_\eps(z)}(z)\, dz$ is uniformly bounded in $L^6$.
	
	For $(z_1,\cdots, z_6)$ in $\supp(\phi)^6$, write $R=R(z_1,\cdots, z_6)$ for the largest $r$ such that the balls $B_{z_i}(r)$ are all disjoint. Then for $\eps<R$, by the domain Markov property and \cref{lem:nested_dmp}, we have that
	\[ \mathbb{E}[\prod_{i=1}^6 \hat h_{r_\eps(z_i)}^\D(z_i)]=\mathbb{E}[\prod_{i=1}^6 \hat h_R^\D(z_i)]. \]
	By repeated application of H\"{o}lder's inequality, the term on the right hand side above is less than $\prod_{i=1}^6(\mathbb{E}[(h^\D_R(z_i))^6])^{1/6}$, and since each $h_R^\D(z_i)$ is Gaussian with variance less than some universal constant times $\log(1/R)$, this is less than a constant times $|\log(R)|^3$. When $R<\eps$, we can similarly bound $\mathbb{E}[\prod_{i=1}^6 \hat h_{r_\eps(z_i)}^\D(z_i)]\le \prod_{i=1}^6(\mathbb{E}[(\hat h^\D_{r_\eps(z_i)}(z_i))^6])^{1/6}\le |\log(\eps)|^3\le |\log(R)|^3$. Thus by expansion we obtain that
	
	\[ \mathbb{E}[\left(\int \hat{h}^\D_{r_\eps(z)}(z) \phi(z) \, dz \right)^6]= C(\phi) \left(1+ \iint_{D^6} |\log(R(z_1,\cdots, z_6))|^3 \, d\mathbf{z} \right) <\infty \]
	where $C(\phi)$ is a finite constant depending on $\phi$ but not $\eps$.
	Since this bound is uniform in $\eps$, the proof is complete.
\end{proof}

\begin{proof}[Proof of \cref{prop:fourth_moment}\,(2)\&(3)]
	Suppose that $\phi_n$ is a sequence of functions in $C_c^\infty(\D)$ converging to $\phi\in C_c^\infty(\D)$. Then by the previous part of this  proof, $$\mathbb{E}[(h^\D,\phi_n)^4]=\lim_{\eps\to 0}\mathbb{E}[(\int_{\D} \phi_n(z) \hat{h}^\D_{r_\eps(z)}(z)\, dz
	)^4]$$ for each $n$, and this expectation is easily seen to be uniformly bounded in $n$ (using H\"{o}lder's inequality and the fact that we know the marginal distributions of the $\hat{h}^\D$'s; as above). By the stochastic continuity assumption, we have that $(h^\D,\phi_n)\to (h^\D,\phi)$ in probability as $n\to \infty$.  Putting this together with the uniform boundedness in $L^4$, we can deduce in particular that $(h^D,\phi_n)$ converges in $L^2$ to $(h^D,\phi)$ as $n\to \infty$. This implies the continuity of $K_2^D$ by the Cauchy--Schwarz inequality.
	
	The same arguments can be used to show that $(h^\D,\phi_n)$ is uniformly bounded in $L^4$ when $\phi_n$ is as in \cref{ass:ci_dmp}(ii). This implies that the convergence of this assumption also holds in $L^2$.
	\end{proof}


\bibliographystyle{abbrv}
\bibliography{EP_bibliography}

\end{document}